\documentclass[12pt]{amsart}
\usepackage{amssymb,amsmath,amsthm,latexsym}
\usepackage{dsfont}
\usepackage{amssymb}
\usepackage{indentfirst}
\usepackage{graphicx}
\usepackage{caption}
\usepackage{multicol}
\usepackage{enumitem}
\usepackage{hyperref}
\usepackage{tikz-cd}
\hypersetup{citecolor=red,}   
\usepackage{amsfonts, array, hhline} 
\usepackage{color}
\usepackage{lmodern}
\usepackage[left=1in,right=1in,top=1in,bottom=1.5in]{geometry}
\usepackage[mathscr]{euscript}
\usepackage{makecell}

\makeatletter
\theoremstyle{definition}
\newtheorem*{rep@theorem}{\rep@title}
\newcommand{\newreptheorem}[2]{%
\newenvironment{rep#1}[1]{%
 \def\rep@title{#2 \ref{##1}}%
 \begin{rep@theorem}}%
 {\end{rep@theorem}}}
\makeatother

\numberwithin{equation}{section}

\theoremstyle{definition}
\newtheorem{dfn}{Definition}[section]

\newtheorem{thm}[dfn]{Theorem}
\newtheorem{lm}[dfn]{Lemma}
\newtheorem{prop}[dfn]{Proposition}
\newtheorem{crl}[dfn]{Corollary}

\newtheorem{question}{Question}
\newreptheorem{lm}{Lemma}

\theoremstyle{remark}
\newtheorem{rmk}[dfn]{Remark}

\newcommand{\pt}{\partial}
\newcommand{\mc}[1]{\mathcal{#1}}
\newcommand{\mf}[1]{\mathfrak{#1}}

\newcommand{\rar}{\rightarrow}
\newcommand{\la}{\langle}
\newcommand{\ra}{\rangle}

\newcommand{\R}{\mathbb{R}}

\newcommand{\C}{\mathbb{C}}
\renewcommand{\H}{\mathbb{H}}

\newcommand{\area}{{\rm area}}

\newcommand{\clconv}{\overline{\rm conv}}

\renewcommand{\area}{{\rm area}}

\newcommand{\ads}{\mathbb{A}d\mathbb{S}}
\renewcommand{\hat}{\widehat}
\renewcommand{\tilde}{\widetilde}
\newcommand{\fff}{\mathrm{I}}
\newcommand{\sff}{\mathrm{I\!I}}
\newcommand{\tff}{\mathrm{I\!I\!I}}
\newcommand{\ms}[1]{\mathscr{#1}}
\newcommand{\psl}{{\rm PSL}}
\renewcommand{\sl}{{\rm SL}}
\renewcommand{\ker}{{\rm Ker}}

\newcounter{notes}%

\newcommand\blfootnote[1]{%
  \begingroup
  \renewcommand\thefootnote{}\footnote{#1}%
  \addtocounter{footnote}{-1}%
  \endgroup
}

\title[]{Rigidity of anti-de Sitter (2+1)-spacetimes with convex boundary near the Fuchsian locus}

\author{Roman Prosanov}
\address{University of Vienna, Faculty of Mathematics, Oskar-Morgenstern-Platz 1, 1090 Vienna, Austria}
\email{roman.prosanov@univie.ac.at}

\author{Jean-Marc Schlenker}
\address{Jean-Marc Schlenker:
University of Luxembourg, Department of Mathematics, 
Maison du nombre, 6 avenue de la Fonte,
4364 Esch-sur-Alzette, Luxembourg}
\email{jean-marc.schlenker@uni.lu}

\begin{document}

\begin{abstract}
We prove that globally hyperbolic compact anti-de Sitter (2+1)-spacetimes with strictly convex spacelike boundary that is either smooth or polyhedral and whose holonomy is close to Fuchsian are determined by the induced metric on the boundary.
\end{abstract}

\maketitle
\blfootnote{MSC Class: 53C50; 53C45; 53C24; 53C30; 57K35; 52A15}
\tableofcontents

\section{Introduction and results}

\subsection{Motivation}

%In the study of surfaces in space, \emph{intrinsic geometry} stands for measuring intrinsic distances on a surface, while \emph{extrinsic geometry} describes how a surface is curved within space. Deep questions arise from attempts to understand how the two are intertwined. A remarkable result states that the extrinsic geometry of a closed convex surface in Euclidean 3-space is completely determined by the intrinsic geometry. Using this fact as the basis of our methodology, we aim to investigate to which extent it can be generalized. 

A deep geometric result states that the shape of a convex body in Euclidean 3-space is completely determined by the intrinsic geometry of its boundary~\cite{Pog}. Using this fact as a foundation for our study, we aim to investigate, to which extent it can be generalized. Standard proofs of this result rely significantly on the fact that a convex body has the topology of a 3-ball. However, it was observed that this phenomenon holds also in nontrivial topology, provided that we stay in the setting of homogeneous geometry. In the case of \emph{hyperbolic geometry}, which is the richest of Riemannian homogeneous geometries in dimension 3, manifestations of this were established, e.g., in the papers~\cite{Sch, Pro2}.

In the present paper, we are concerned with \emph{anti-de Sitter geometry}, which is a Lorentzian cousin of hyperbolic geometry. Namely, we deal with \emph{globally hyperbolic Cauchy compact anti-de Sitter (2+1)-spacetimes}. They appear, e.g., in models of quantum gravity, see the fundamental paper of Witten~\cite{Wit}. They are reminiscent in their properties to \emph{quasi-Fuchsian hyperbolic 3-manifolds}. In particular, it is conjectured that globally hyperbolic compact anti-de Sitter (2+1)-spacetimes with spacelike convex boundary are determined by the intrinsic geometry of the boundary. See, e.g.,~\cite[Questions 3.3--3.6]{BBD+}. Every such spacetime has two boundary components and can be canonically extended beyond the boundary to a maximal spacetime. Fundamentally, the conjecture means that the geometry of the whole spacetime is determined by the intrinsic geometry of two of its spacelike slices provided that the slices are convex in the opposite directions. In this article, we prove the conjecture among spacetimes whose holonomy is close to Fuchsian and whose boundary is either smooth or polyhedral. In the works~\cite{BMS, Pro4}, some other partial cases of the conjecture are established, see below for details.
%The rigidity of convex surfaces, together with some similar results observed for hyperbolic 3-manifolds, with which anti-de Sitter spacetimes have some similarity, make it reasonable to expect that when such a spacetime is compact and has convex boundary, then it is determined by the intrinsic geometry of the boundary. Fundamentally this means that if the intrinsic geometry of two spacelike slices of a spacetime determine the geometry of the whole spacetime provided that the slices are convex. (Note that a compact spacetime with convex boundary can be uniquely extended beyond its boundary to a maximal spacetime.) In the present paper, we prove this principle for spacetimes whose holonomy is close to Fuchsian, and whose boundary is either smooth or polyhedral. 

\subsection{Parallels between hyperbolic and anti-de Sitter geometry}
\label{parallel}

For the whole manuscript, $S$ is a closed orientable surface of genus $k \geq 2$. In what follows, we take \emph{GHC AdS (2+1)-spacetimes} to mean globally hyperbolic compact anti-de Sitter (2+1)-spacetimes with spacelike convex boundary and \emph{GHMC AdS (2+1)-spacetimes} to mean globally hyperbolic maximal Cauchy compact anti-de Sitter (2+1)-spacetimes. The exact definitions are given in Section~\ref{ghmc}. The former ones are homeomorphic to $S \times [-1, 1]$ and the latter ones are homeomorphic to $S \times \R$. 

The similarity of GHMC AdS (2+1)-spacetimes to quasi-Fuchsian hyperbolic 3-manifolds was discovered by Mess in a groundbreaking paper~\cite{Mes}. In particular, Mess proved that the marked moduli space of GHMC AdS structures on $S\times \R$ is parameterized by the product of two copies of the Teichm\"uller space of $S$, see Section \ref{ghmc}. This is strongly reminiscent of the Bers Double Uniformization Theorem for quasi-Fuchsian hyperbolic 3-manifolds~\cite{Ber}.

The GHMC AdS (2+1)-spacetimes and quasi-Fuchsian hyperbolic 3-manifolds are also similar in that in both cases they contain a nonempty inclusion-minimal compact totally convex subset, \emph{the convex core}. Here, a subset is \emph{totally convex} if it contains every geodesic segment between every pair of its points. The convex core is either a totally geodesic surface or it has nonempty interior and its topological boundary is the disjoint union of two closed surfaces whose intrinsic metric is hyperbolic and that are ``pleated'' along a measured geodesic lamination. 

In the hyperbolic case, Thurston conjectured that the corresponding marked moduli space can be parameterized both by the induced metrics on the boundary of the convex core and by the pleating laminations. These questions have two aspects: the realization part -- to describe the image of the respective map from the marked moduli space, and the rigidity part -- to show that the map is injective. For the first conjecture, the realization part follows, e.g., from the work of Labourie~\cite{Lab}; the rigidity is only known to hold for an open dense set, as shown in the work~\cite{Pro2} of Prosanov. For the second conjecture, the image was described in~\cite{BO} by Bonahon--Otal; the rigidity has been obtained only recently in~\cite{DS} by Dular--Schlenker. (An important earlier advance was made by Bonahon in~\cite{bonahon}, where he proved the rigidity in a neighborhood of the Fuchsian locus.) Prior to that, it was only known in a related, though slightly different setting when $S$ is the once-punctured torus, due to Series~\cite{Ser}.

Mess conjectured that similar parameterizations hold for GHMC AdS (2+1)-spacetimes. In this case, the realization of the induced metrics was established by Diallo, published as an appendix to paper~\cite{BDMS}. The realization of pleating laminations was shown in~\cite{BS3} by Bonsante--Schlenker, where the rigidity near the Fuchsian locus was also deduced. Except that, the rigidity parts of these problems remain open.
%This remarkable analogy suggests to call GHMC AdS spacetimes ``quasifuchsian'', following the hyperbolic terminology.

%%It also suggests exploring to what extent results known (or conjectured) for quasifuchsian hyperbolic manifolds can be extended to quasifuchsian AdS spacetimes, as already done by Mess \cite{mes} for the induced metric and measured lamination on the boundary of the convex core.

\subsection{Manifolds with boundary}

As we saw, geometric rigidity questions are interconnected with the corresponding realization questions. In the case of Euclidean convex bodies, the relevant realization questions were answered by Nirenberg~\cite{Nir} in the smooth case, and by Alexandrov~\cite{Ale} in the polyhedral and, furthermore, in the general case of metric structures of curvature $K\geq 0$ in the sense of Alexandrov, without any additional regularity assumption.

The analogy of Section~\ref{parallel} suggests exploring to what extent other remarkable properties of hyperbolic 3-manifolds extend to AdS (2+1)-spacetimes. In the hyperbolic case, Thurston also conjectured the following beautiful fact, established by Labourie in~\cite{Lab}: given a 3-dimensional manifold $M$ that admits a (non-complete) hyperbolic metric $m$ for which the boundary is locally convex, then any Riemannian metric on $\pt M$ of curvature $\kappa >-1$ can be realized as the induced intrinsic metric for some choice of $m$. In~\cite{Sch}, Schlenker proved that $m$ is uniquely determined by the boundary metric, provided that the boundary is smooth and strictly convex, as is the case for Labourie's realizations. Here, the boundary is said to be \emph{strictly convex} if the shape operator is positive definite. 

%Here we are particularly interested in the possible generalization to AdS of a specific property of hyperbolic 3-manifolds: given a 3-dimensional manifold $M$ that admits a (non-complete) hyperbolic metric $g$ for which the boundary is locally convex, this metric $g$ is uniquely determined by the induced metric on the boundary, which can be any metric of curvature $K>-1$. (This is one way to state \cite[Theorem 0.1]{hmcb}.)

From now on, $M$ is a 3-manifold diffeomorphic to $S \times [-1,1]$. (The diffeomorphism is fixed for the whole paper.) The anti-de Sitter analog that we are interested in is the following question.

\begin{question} \label{q:main}
Let $s_+, s_-$ be two Riemannian metrics on $S$ of curvature $\kappa <-1$. Is there a unique GHC AdS structure on $M$ such that the boundary is smooth and strictly convex, the induced metric on its future boundary is $s_+$, and the induced metric on its past boundary is $s_-$? 
\end{question}

The existence part of this statement was proved by Tamburelli \cite{Tam}. However, the uniqueness remains elusive in general. In~\cite{BMS}, Bonsante--Mondello--Schlenker showed the uniqueness for the case when $s_+$, $s_-$ have constant curvatures $\kappa_+, \kappa_- <-1$ respectively, satisfying $\kappa_+=-\kappa_-/(\kappa_-+1)$.

In fact, a key open question is, given such a spacetime $M$, %% an AdS structure $g$ on $S\times \R$ with smooth, strictly convex boundary,
whether any first-order deformation of the AdS structure on $M$ that leaves invariant the induced metric on the boundary is trivial. We will discuss this further below.

There is a polyhedral counterpart to this story, when $s_+, s_-$ are so-called \emph{cone-metrics}. There are, however, some special peculiarities, since the most natural notion of polyhedral boundary (locally modeled on polyhedral sets) turns out to be insufficient to capture all the picture of what is happening. Some singularities, similar to those of pleating laminations, may appear, which obstruct the analysis of this case. We refer to~\cite{Pro4} of Prosanov for a description, and to~\cite{Pro2} also of Prosanov for a discussion of these phenomena in the hyperbolic setting. In~\cite{Pro4}, a polyhedral realization result for all pairs of cone-metrics satisfying the appropriate curvature condition is derived. Furthermore, the rigidity is established, provided that the cone-metrics are sufficiently small in some sense.
In~\cite{Pro2}, the corresponding hyperbolic realization statement is obtained, and the rigidity is shown in a generic case. Luckily, the mentioned singularities lie outside of the scope of the present article, since they happen away from the Fuchsian locus.

Note that both this paper and~\cite{Pro4} exploit the fact that anti-de Sitter geometry is a subgeometry of projective geometry. However, the other tools are quite different. In the present article, we apply the infninitesimal Pogorelov map and obtain local rigidity near the Fuchsian locus. In~\cite{Pro4}, the technique of geometric transitions is employed, which allows to deduce global rigidity for small enough metrics on the boundary.

%% \subsection*{Notations}

\subsection{Our results}

Let us fix some notations to state the main results below. Let $\ms M$ be the space of GHC AdS structures on $M$ with smooth strictly convex boundary, considered up to isotopy. (The topology on $\ms M$ is discussed in Section~\ref{secembed}.) There is a distinguished subspace $\ms M_0 \subset \ms M$ of \emph{Fuchsian} structures, i.e., the ones with Fuchsian holonomy. We will give an exact definition in Section~\ref{secembed}. For now, we remark that, equivalently, $m$ is Fuchsian if and only if $(M, m)$ contains an embedded totally geodesic surface isotopic to $S \times \{0\}$. We also denote by $\ms S$ the space of smooth metrics on $S$ with curvature $\kappa <-1$, considered up to isotopy.

The Gauss equation in anti-de Sitter 3-space $\ads^3$ shows that a smooth spacelike strictly convex surface in $\ads^3$ has an induced metric with curvature $\kappa <-1$. Therefore, there is a natural map $\ms I:\ms M \to \ms S\times \ms S$ that sends a GHC AdS structure on $M$ to the induced metrics on the two connected components of its boundary. A strengthening of Question \ref{q:main} is whether $\ms I$ is a homeomorphism. The following statement is a partial answer to it in a neighborhood of Fuchsian structures.

\begin{thm}
\label{main*}
There exists a neighborhood $U$ of $\ms M_0$ in $\ms M$ and an open subset $U'\subset \ms S\times \ms S$, which contains the diagonal, such that $\ms I|_U$ is a homeomorphism from $U$ to $U'$.
\end{thm}

We want to stress the rigidity part of Theorem~\ref{main*}.

\begin{crl} \label{main}
 There exists a neighborhood $U$ of $\ms M_0$ in $\ms M$ such that, for any $m_1, m_2 \in U$, if there exists a map $f: (M, m_1) \rar (M, m_2)$ isotopic to the identity such that $f|_{\pt M}$ is an isometry, then $f$ is isotopic to an isometry.
\end{crl}

% Let $\ms H$ be the space of Riemannian metrics on $\pt M$ of curvature $<-1$, considered up to isotopies on $M$. We have the realization map $\ms F:\ms G \rar \ms H$. 

In Section \ref{sc:polyhedral} we consider analogous statements for GHC AdS (2+1)-spacetimes with strictly polyhedral boundary. Here, the notion of being \emph{strictly polyhedral} is a polyhedral analog to strict convexity, see Section \ref{sc:polyhedral} for the definitions. We denote by $V$ a finite subset of $\partial M$, and denote by $V_+$ and $V_-$ its subsets belonging to each boundary components. We assume that both $V_+$ and $V_-$ are nonempty. We then denote by $\ms M(V)$ the space of GHC AdS structures on $M$ with strictly polyhedral boundary and with vertices at $V$. We also denote by $\ms M_0(V)$ the subset of Fuchsian structures, defined the same as in the smooth case.

For each of those AdS structures in $\ms M(V)$, the induced metric on $\partial M$ is hyperbolic, with cone-singularities of angle larger than $2\pi$ at the points of $V$. We consider $V_\pm$ as subsets of $S$ and denote by $\ms S(V_\pm)$ the spaces of hyperbolic metrics with cone-singularities at $V_\pm$ respectively of angles $>2\pi$. We have the induced metric map $\ms I_V: \ms M(V) \rar \ms S(V_+) \times \ms S(V_-)$. We then have the following results.

\begin{thm}
\label{mainp*}
There exists a neighborhood $U$ of $\ms M_0(V)$ in $\ms M(V)$ and an open subset $U'\subset \ms S(V_+)\times \ms S(V_-)$, which contains the diagonal, such that $\ms I_V|_U$ is a homeomorphism from $U$ to $U'$.
\end{thm}

\begin{crl}
\label{mainp}
There exists a neighborhood $U$ of $\ms M_0(V)$ in $\ms M(V)$ such that for any $m_1, m_2 \in U$, if there exists $f: (M, m_1) \rar (M, m_2)$ isotopic to the identity such that $f|_{\pt M}$ is an isometry, then $f$ is isotopic to isometry.
\end{crl}

\subsection{A word on the proofs}

The main issue in answering Question \ref{q:main} lies in an {\em infinitesimal rigidity} statement, which can be formulated as follows: Is the differential of $\ms I$ always injective? This can also be stated as follows.

\begin{question}\label{q:infrig}
  Let $m\in \ms M$, and $(s_+, s_-)=\ms I(m)$ be the pair of the induced metrics on its boundary. Let $\dot m$ be a non-trivial first-order deformation of $m$. Is the corresponding variation $(\dot s_+, \dot s_-)$ of the induced metrics on the boundary necessarily nonzero? 
\end{question}

Note that a first-order deformation $\dot m$ can include a deformation of the boundary of $(M,m)$, but also a deformation of its holonomy representation.

%Once such a statement is known, it is possible to apply the Inverse Function Theorem and conclude that the map sending an AdS structure to the induce metric on the boundary is a local homeomorphisms, and therefore obtain theorems \ref{main*} and \ref{mainp*}.

Unfortunately, we do not know how to prove this infinitesimal rigidity statement in general. For spacetimes with Fuchsian holonomy, however, there is a ``trick'' that makes it possible. Namely, we use the infinitesimal Pogorelov map (see Section \ref{ssc:pogorelov}) to ``translate'' this infinitesimal rigidity question into another one concerning globally hyperbolic Minkowski spacetimes, also with Fuchsian holonomy.

We are then in a position to use infinitesimal rigidity statements proved in this Minkowski setting, by Smith \cite{Smi} when the boundary is smooth and by Fillastre--Prosanov \cite{FP} when the boundary is polyhedral.

In~\cite{bonahon}, Bonahon proved the rigidity part of Thurston's pleating lamination conjecture for quasi-Fuchsian hyperbolic 3-manifolds near the Fuchsian locus. A similar strategy was later applied by Bonsante--Schlenker in~\cite{BS3} in the respective anti-de Sitter situation. There, the infinitesimal rigidity at the Fuchsian locus was translated to a statement in the language of pure Teichm\"uller theory; no use of the Pogorelov map was necessary. The rest in~\cite{bonahon, BS3} was accomplished by interpreting the infinitesimal rigidity as a statement on transversality of some associated manifolds and then by exploring the special structure of the space of measured laminations. In our smooth case, the last part has to be replaced with some arguments from functional analysis and an application of the Nash--Moser inverse function theorem. %In the polyhedral case, we conclude from differentiability properties of $\ms I_V$.

\subsection{Related results}

There is a vast body of related results in this area of geometry. Here we survey some of them that we have not mentioned yet, though our exposition does not aim to be exhaustive. For other versions of such a survey, we refer to~\cite{FS4} by Fillastre--Smith and to~\cite{Sch5} of Schlenker.

We rely on a resolution of Question~\ref{q:main} for Fuchsian structures due to Labourie--Schlenker~\cite{LS} in the smooth case and due to Fillastre~\cite{Fil} in the polyhedral case. The hyperbolic polyhedral case was resolved by Fillastre in an earlier paper~\cite{Fil2}, while the realization part of the hyperbolic smooth case is due to Gromov~\cite{Gro2}. We note that the methods of these papers do not allow to step into a neighborhood of the Fuchsian locus.

It is natural to wonder whether there exists a common generalization of the smooth and the polyhedral cases. For convex bodies in Euclidean 3-space, this question was a starting point for the development by Alexandrov of theory of Alexandrov spaces~\cite{Ale}, now a cornerstone of modern geometry. The rigidity problem for convex bodies with general boundary was resolved by Pogorelov~\cite{Pog}, and is known to be notoriously difficult. In the setups relevant to us, some realization results for general convex boundary are~\cite{Slu, FS6, Lab2} by Slutskiy, Fillastre--Slutskiy and Labeni respectively. The only known rigidity result for nontrivial topology and general convex boundary is~\cite{Pro} by Prosanov for Fuchsian hyperbolic 3-manifolds.

Up to scaling, there are three types of spacetimes of constant curvature: anti-de Sitter, Minkowski and de Sitter. (Note that in dimension (2+1), solutions of Einstein's equations necessarily have constant curvature.) Mess also classified GHMC Minkowski (2+1)-spacetimes in~\cite{Mes}, while the classification of the de Sitter ones was completed by Scannell in~\cite{Sca}. Questions similar to ours can also be formulated in these settings, though they have to be stated not in the sense of extending metrics from the boundary but rather in the sense of simultaneous isometric embeddings of pairs of surfaces into pairs of spacetimes with the same holonomy. In the Minkowski setting, the smooth case is due to Smith~\cite{Smi} and the polyhedral case is due to Fillastre--Prosanov~\cite{FP}. The de Sitter setting turns our to be dual in some sense to the hyperbolic setting and is covered by the papers~\cite{Sch} and~\cite{Pro3} of Schlenker and Prosanov. It is interesting to note that the toolboxes used differ from setting to setting and between the smooth and the polyhedral cases.

Question~\ref{q:main} and its hyperbolic counterpart can be interpreted as a realization and rigidity statements for convex domains in anti-de Sitter 3-space or in hyperbolic 3-space that are invariant with respect to a discrete and faithful representation of a surface group. It is natural to wonder about the validity of universal versions of such results, i.e., realizations of non-invariant metrics without group actions. Clearly, for the rigidity to hold, one must impose a kind of gluing condition at infinity. We refer to a survey~\cite{Sch4} of Schlenker and to papers~\cite{BDMS, CS2, MS, Mes2} of Bonsante--Danciger--Maloni--Schlenker, Chen--Schlenker, Merlin--Schlenker and Mesbah for some advances on the respective realization problems. The rigidity, while expected to hold, remains widely open.

The classical result of Bers~\cite{Ber} concerns the prescription of conformal structures at infinity for quasi-Fuchsian hyperbolic 3-manifolds. One can, however, define also measured foliations at infinity, which are analogs at infinity of the pleating laminations of the convex core. In this setting, an analog of Bonahon's result~\cite{bonahon} was obtained in~\cite{Cho} by Choudhury.

Question~\ref{q:main} is also equivalent to a question of simultaneous isometric embeddings of two metrics on $S$ into GHMC AdS (2+1)-spacetimes. Another interesting problem is to attempt to embed only one metric into a spacetime with prescribed left or right metric in the sense of Mess. This problem in the smooth case was fully resolved in~\cite{Tam} by Tamburelli (realization) and in~\cite{CS} by Chen--Schlenker (rigidity). In the polyhedral case, the realization was shown in~\cite{Pro4} by Prosanov, as well as the rigidity for sufficiently small metrics. In the Minkowski setting, instead of left or right metric it is natural to prescribe the linear part of the holonomy. In the smooth case, this follows from the work~\cite{TV} by Trapani--Valli, while in the polyhedral case, it is proved in~\cite{FP} by Fillastre--Prosanov. Actually, we will employ some of these results in our proof. We note that analogous questions in the de Sitter and in the hyperbolic settings (when the left or right metric is replaced by a conformal structure at infinity) are fully open.

\vskip+0.2cm

\textbf{Acknowledgments.} The first author is very grateful to Fran\c{c}ois Fillastre, Thomas K\"orber, Andrea Seppi and Graham Smith for useful discussions and to Michael Eichmair for helpful comments and his support. We are very thankful to the anonymous referee for plenty of valuable remarks.

This research of the first author was funded in whole by the Austrian Science Fund
(FWF) https://doi.org/10.55776/ESP12. For open access purposes, the author has applied a CC BY public copyright license to any author-accepted manuscript version arising
from this submission.

\vskip+0.2cm

\textbf{Notations.}
For the reader convenience, we list the main notations used in the smooth part of the paper. (The notation in the polyhedral setting is similar.)

\medskip
\begin{center}
\begin{tabular}[h]{|c|c|}
  \hline
  $S$ & a closed surface of genus $k \geq 2$ \\
  \hline
  $M$ & 3-manifold diffeomorphic to $S\times [-1,1]$ \\
  \hline
  $\mc S$ & the space of Riemannian metrics on $S$ of curvature $<-1$\\
  \hline
  $\mc M$ & \makecell{the space of GHC AdS structures on $M$ }\\
        \hline
   $G$ & ${\rm PSL}(2, \R) \times {\rm PSL}(2, \R)$\\
   \hline
   $\mc R$ & the space of representations $\pi_1S \rar G$ up to conjugation \\
   \hline
   $\mc E_\pm$ & \makecell{equivariant strictly future-/past-convex embeddings of $S$ \\ into $\ads^3$ up to isometry} \\
   \hline
   $\mc E$ & \makecell{pairs of embeddings in $\mc E_+ \times \mc E_-$ \\ equivariant with respect to the same representation}\\
   \hline
   $\mc{GC}_\pm(s)$ & \makecell{the spaces of Gauss--Codazzi fields for a metric $s$ on $S$}\\
   \hline
   $\mc I$ & the induced metric map $\mc E \rar \mc S \times \mc S$ \\
   \hline
   $\ms T$ & the Teichm\"uller space of $S$ \\
   \hline
\end{tabular}
\end{center}

We add the subscript $F$ for subspaces of \emph{Fuchsian} objects, i.e., related to the representations of $\pi_1S$ into $G_F := {\rm PSL}(2,\R)$. We add the superscript $0$ when anti-de Sitter space is replaced by Minkowski space. We use the font $\ms S$, $\ms M$, etc., for the quotients of $\mc S$, $\mc M$, etc., by isotopies. We use the hat: $\hat{\mc R}$, $\hat{\mc E}$, etc., for spaces {\bf not} up to conjugation/isometries.  %% introduction

\section{Preliminaries}
\label{prelim}

\subsection{Anti-de Sitter geometry}
\label{secads}

The main reference for us is the excellent survey~\cite{BS2} by Bonsante--Seppi. Other good sources include the original paper of Mess~\cite{Mes}, an accompanying paper~\cite{Mes+}, a survey~\cite{Bar2} of Barbot and a survey~\cite{FS4} of Fillastre--Smith.

First we recall the quadric model for anti-de Sitter 3-space. Let $\R^{2,2}$ be the 4-dimensional real vector space equipped with a symmetric bilinear form of signature $(2,2)$: 
\[\la x, y \ra=x_1y_1+x_2y_2-x_3y_3-x_4y_4.\]
Anti-de Sitter 3-space is modeled on the following quadric
\[\H^{2,1}:=\{x \in \R^{2,2}:~\la x, x \ra=-1\}.\]
The induced metric is a Lorentzian metric of constant curvature $-1$. 

We will, however, consider as the main model for anti-de Sitter 3-space the projective quotient of $\H^{2,1}$. We will denote it by $\ads^3$ and consider it as a subset of $\mathbb{RP}^3$. Recall that it is a minimal model space for anti-de Sitter geometry, i.e., any other Lorentzian 3-space of curvature $-1$ with maximal isometry group is a covering of $\ads^3$, see~\cite{BS2}. 

We also briefly recall the $\psl(2, \R)$-model for $\ads^3$, which is a special feature of dimension 3. Note that on the space of real 2$\times$2-matrices the determinant is a quadratic form of signature $(2,2)$. Considering the polarization of the minus determinant as a bilinear form, we can identify this space with $\R^{2,2}$. Under this identification, $\sl(2, \R)$ gets identified with $\H^{2,1}$, and $\psl(2,\R)$ gets identified with $\ads^3$. 

This model admits a straightforward description of the isometry group of $\ads^3$. We are interested only in its identity component, which we denote by $G$. Since for $A, B, X \in \sl(2,\R)$ we have $\det(X)=\det(AXB^{-1})$, the left and right multiplications by the elements of $\psl(2, \R)$ are isometries. This produces a map $\psl(2, \R) \times \psl(2,\R) \rar G$. A dimension count shows that this is actually an isomorphism
\begin{equation}
\label{g}
G \cong \psl(2, \R) \times \psl(2,\R).
\end{equation}

We fix an orientation and a time-orientation on $\ads^3$. The latter means that we pick a (nowhere zero) timelike vector field $\xi$ on $\ads^3$ to define future and past. Namely, for every $x \in \ads^3$ the subset of timelike vectors of $T_x\ads^3$ has two connected components. The vectors from the same component as $\xi(x)$ are called \emph{future} vectors, the others are called \emph{past}.

\subsection{Quasi-Fuchsian representations}
\label{ssc:quasifuchsian}

%\marginnote{JM: ref for this?}

%A globally hyperbolic spacetime $N$ is \emph{maximal} if for every isometric embedding $\iota: N \rar N'$ sending a Cauchy surface of $N$ to a Cauchy surface of $N'$, $\iota$ is onto. A \emph{GHMC spacetime} is a globally hyperbolic maximal spacetime, whose Cauchy surface are compact. (Note that all Cauchy surfaces in a spacetime are homeomorphic to each other.)

%Mess famously classified anti-de Sitter GHMC (2+1)-spacetimes in~\cite{Mes}. We need to recall some pieces of this classification.

To proceed further, we need to recall the basics of Teichm\"uller theory. The \emph{Teichm\"uller space} $\ms T$ is the space of isotopy classes of hyperbolic metrics on $S$. Denote by $\H^2$ the hyperbolic plane and by $G_F$ the identity component of its isometry group. Every hyperbolic metric on $S$ is a $(G_F, \H^2)$-structure in the sense of Thurston. For an introduction to Thurston's approach to geometric structures (later we will also apply this framework to anti-de Sitter structures), we refer to Thurston's notes~\cite{Thu} and an accompanying paper~\cite{CEG}. Provided that $S$ and $\H^2$ are considered oriented, the holonomies of orientation-preserving hyperbolic structures form a connected component in the representation space of $\pi_1S$ to $G_F$. We call such representations \emph{Fuchsian} and denote by $\mc R_F$ the quotient of this component by conjugation of $G_F$, the component itself is denoted by $\hat{\mc R}_F$. In terms of notation, we will not distinguish between representations and their classes. It is well-known that the map sending a hyperbolic structure to its holonomy provides an identification $\ms T \cong \mc R_F$. See, e.g., the book of Farb--Margalit~\cite{FM} for an introduction. 

Consider $G_F \cong \psl(2, \R)$ as the diagonal subgroup of $G \cong \psl(2,\R) \times \psl(2,\R)$. We call a representation $\rho: \pi_1 S \rar G$ \emph{quasi-Fuchsian}
%\marginnote{JM: we need to choose between ``quasifuchsian'' and ``quasi-Fuchsian''.} 
if it is a deformation of a Fuchsian representation.
%\marginnote{JM: Maybe we should point out that this corresponds to the AdS quasifuchsian reps, it might be surprising for readers used to the hyperbolic quasifuchsian theory.}
By definition, quasi-Fuchsian representations form a connected component of a representation space. We denote by $\mc R$ the quotient by conjugation of $G$, the component itself is denoted by $\hat{\mc R}$. The projections of a quasi-Fuchsian representation to both factors in~(\ref{g}) are discrete and faithful. Hence, we have $\hat{\mc R}\cong \hat{\mc R}_F \times \hat{\mc R}_F$ and
\begin{equation}
\label{r}
\mc R \cong \mc R_F \times \mc R_F \cong \ms T \times \ms T.
\end{equation}
For a given $\rho \in \mc R$, identification~(\ref{r}) produces two hyperbolic metrics on $S$ (given up to isotopy). They are called \emph{the left and the right metrics} of $\rho$.%\marginnote{JM: should we un-comment the part below with the Mess theorem stating that the space of quasifuchsian AdS spacetimes is parameterized by $T\times T$ through the left and right representations?}

The quasi-Fuchsian representations in this sense are exactly the holonomies of GHMC AdS (2+1)-spacetimes, as we explain in the next subsection. We point out a slight discrepancy with the classical setting of 3-dimensional hyperbolic geometry that was discussed in the introduction, where $G_F \cong {\rm PSL}(2,\R)$ is considered naturally embedded into ${\rm PSL}(2,\C)$ and only quasi-conformal deformations of Fuchsian representations are called quasi-Fuchsian. They then form the interior of the respective component of the representation space. There are also representations at the boundary, not quasi-Fuchsian in this sense, which do not have counterparts in the anti-de Sitter setting.

%Any GHMC (2+1)-spacetime is diffeomorphic to $S \times (-1,1)$, and hence $\pi_1N=\pi_1S$. Mess showed that quasi-Fuchsian representations are exactly the holonomies of anti-de Sitter GHMC (2+1)-spacetimes. 
%For a given such spacetime $N$, the class of its holonomy produces, using~(\ref{r}), two classes of hyperbolic metrics on $S$. They are called \emph{the left and the right metrics} of $N$. 

\subsection{Globally hyperbolic anti-de Sitter spacetimes}
\label{ghmc}

An \emph{anti-de Sitter (2+1)-spacetime} is a 3-manifold equipped with an anti-de Sitter metric, i.e., with a Lorentzian metric of curvature $-1$. Provided that it is oriented and time-oriented, it is equipped with a $(G, \ads^3)$-structure in the sense of Thurston. Hence, every anti-de Sitter (2+1)-spacetime $M$ is equipped with a holonomy representation $\rho: \pi_1M \rar G$ and a $\rho$-equivariant developing map $f: \tilde M \rar G$, which are determined up to an isometry of $\ads^3$. A (2+1)-spacetime is \emph{globally hyperbolic} if it admits a \emph{Cauchy surface}, i.e., a surface that is intersected exactly once by every inextensible causal curve. All Cauchy surfaces in a spacetime are homeomorphic to each other. See, e.g., the book~\cite{ONe} by O'Neill for an introduction to Lorentzian geometry. By a result of Geroch~\cite{Ger}, any globally hyperbolic (2+1)-spacetime admits a parametrization $S \times \R$, in which all fibers $S \times \{r\}$ are Cauchy surfaces.

A globally hyperbolic spacetime $N$ is \emph{maximal} if every isometric embedding $N \rar N'$ that sends a Cauchy surface of $N$ to a Cauchy surface of $N'$ is onto. It is \emph{Cauchy compact} if its Cauchy surfaces are compact. Recall that we refer to globally hyperbolic maximal Cauchy compact anti-de Sitter spacetimes as GHMC AdS for short. Any GHMC (2+1)-spacetime $N$ is diffeomorphic to $S \times \R$, and hence $\pi_1N=\pi_1S$. Let now $N$ be a manifold diffeomorphic to $S \times \R$, and let $\ms N$ be the space of isotopy classes of GHMC AdS structures on $N$, with the topology induced by uniform $C^\infty$-convergence of developing maps on compact subsets of $\tilde N$.
%%topology defined similarly to above.
%natural topology defined below in Section \ref{secembed}.
Mess showed in~\cite{Mes} that the holonomy of every GHMC AdS (2+1)-spacetime is quasi-Fuchsian. Furthermore, it follows from his work that the map $\ms N \rar \mc R$ sending a GHMC AdS structure to its holonomy is a homeomorphism. 
Any spacelike surface $\Sigma \subset \ads^3$ invariant with respect to $\rho \in \mc R$ belongs to the \emph{domain of dependence} $D(\rho) \subset \ads^3$, a convex open $\rho$-invariant set, whose quotient by $\rho$ produces the respective GHMC AdS spacetime.%\marginnote{JM: remains to explain relation to previous subsection.}

%For a given such spacetime $N$, the class of its holonomy produces, using~(\ref{r}), two classes of hyperbolic metrics on $S$. They are called \emph{the left and the right metrics} of $N$. 

\subsection{Spacetimes with boundary and spaces of embeddings}
\label{secembed}

Recall that we use GHC AdS (2+1)-spacetimes to mean globally hyperbolic compact anti-de Sitter (2+1)-spacetimes with spacelike convex boundary.
Let $\mc M$ be the space of GHC AdS structures on $M$ with smooth strictly convex boundary. We pick up the viewpoint of locally homogeneous geometry, already mentioned above, and consider elements of $\mc M$ as $G$-conjugacy classes of pairs $(\rho, f)$ of holonomies $\rho: \pi_1M=\pi_1S \rar G$ and of $\rho$-equivariant developing maps $f: \tilde M \rar \ads^3$. We endow it with the topology of uniform $C^\infty$-convergence on compact subsets of $\tilde M$. Let $\ms M$ be the quotient of $\mc M$ by isotopies. A structure is \emph{Fuchsian} if its holonomy is conjugate into $G_F$.
%\marginnote{JM: so, by pre-composition of $f$ by isotopies of $M$, lifted to $\tilde M$? Should we also quotient by left composition by global isometries?}
%As it follows from the work of Mess~\cite{Mes}, a developing map in this case is a diffeomorphism onto the image.

We will, however, choose a different space to focus our consideration on. To this purpose, we will now notice that an element of $\ms M$ is determined by the ``full'' structure of the boundary, that is, by a lift of its two boundary components to equivariantly embedded complete surfaces in $\ads^3$. Our eventual goal in this paper is to show that, in certain cases, only a part of this data, namely the induced metrics, is sufficient to determine $\ms M$.

We say that an immersion of a surface to $\ads^3$ is \emph{future-convex} (resp. \emph{past-convex}) if it is smooth, spacelike, locally convex, and for each point its future cone (resp. its past cone) locally belongs to the convex side. A strictly future-/past-convex immersion is a one that is future-/past-convex and strictly convex.
Let $\hat{\mc E}_+$ be the space of strictly future-convex orientation-preserving immersions $f: \tilde S \rar \ads^3$ equivariant with respect to a representation $\rho: \pi_1S \rar G$. We endow it with the topology of uniform $C^\infty$-convergence on compact subsets of $\tilde S$. As it follows from the work of Mess~\cite{Mes}, any arising representation is quasi-Fuchsian and any such immersion is an embedding. Let $\mc E_+$ be the quotient of $\hat{\mc E}_+$ by $G$. Similarly, let $\hat{\mc E}_-$ and $\mc E_-$ be the spaces of equivariant strictly past-convex orientation-preserving embeddings and of their classes. Denote by $\hat{\mc E}$ the subset of $\hat{\mc E}_+\times\hat{\mc E}_-$ consisting of pairs of embeddings that are equivariant with respect to the same representation, and by $\mc E$ denote its quotient by $G$. Finally, let $\ms E$ be the quotient of $\mc E$ by isotopies (actually, by pairs of isotopies of $S$, as for this quotient we consider pairs of maps $f_+$ and $f_-$, representing $e \in \mc E$, separately).

We claim that $\ms M$ is naturally homeomorphic to $\ms E$. Indeed, from the description of the spaces it follows that there is a natural continuous injective map $\iota: \ms M \rar \ms E$, sending a GHC AdS structure on $M$ with smooth and strictly convex boundary to the (equivalence class under $G$ of the) pair of equivariant embeddings $f_\pm:\tilde S\to \ads^3$ of its boundary components. It is also clear that $\iota$ is surjective. It only remains to check that $\iota^{-1}$ is continuous.
%%To this purpose, we need to recall the notion of a \emph{GHMC spacetime}.

Let $(\rho_i, f_{i, +}, f_{i,-})$ be a sequence in $\hat{\mc E}$ converging to $(\rho, f_+, f_-)$. The convergence of $\rho_i$ implies that $\rho_i$ and $\rho$ determine on $N$ GHMC AdS structures $n_i$ uniformly $C^\infty$-converging on compacts to a GHMC AdS structure $n$. Furthermore, the convergence of $f_{i,+}$ and $f_{i,-}$ implies that the pairs of surfaces in $N$ corresponding to $(f_{i,+}, f_{i,-})$ converge to the ones corresponding to $(f_+, f_-)$. By associating the part of $N$ bounded by these pairs of surfaces with $M$, one can deduce that $(\rho_i, f_{i, +}, f_{i,-})$ determine on $M$ GHC AdS structures that converge to such a structure determined by $(\rho, f_+, f_-)$. 

Hence, we see that $\iota$ is a homeomorphism. From now on we forget about the spaces $\mc M$ and $\ms M$ and work only with $\mc E$ and $\ms E$. 

%\subsection{Another statement of Theorem \ref{main}}

We denote by $\mc S$ the space of Riemannian metrics on $S$ of curvature $<-1$, and denote by $\ms S$ its quotient by isotopies.
We have the respective induced metric map $\mc I: \mc E \rar \mc S \times \mc S$, descending to a map $\ms E \rar \ms S \times \ms S$, which from now on we will denote by $\ms I$. We call \emph{Fuchsian} elements of $\hat{\mc E}$ whose representation is in $G_F$. This notion descends to elements of $\mc E$ and $\ms E$. We denote by $\hat{\mc E}_F \subset \hat{\mc E}$ the subspace of Fuchsian elements, and denote by $\mc E_F$ its quotient by conjugation. What we will actually prove is

\begin{thm}
\label{main**}
There exists a neighborhood $U$ of $\mc E_F$ in $\mc E$ such that the image of $U$ by $\mc I$ is open, contains the diagonal and $\mc I|_U$ is a homeomorphism onto its image.
\end{thm}

Note that the map $\mc I$ is equivariant with respect to isotopies. Thus Theorem~\ref{main**} implies a similar result for the map $\ms I$ and, through our discussion, implies Theorem~\ref{main*}. From now on we will actually also forget about the spaces $\ms E$ and $\ms S$.

%We note that when we need the space of Riemannian metrics on $S$ of curvature $<-1$ (i.e., not of pairs), we will be denoting it by $\mc S_+$. However, we will be denoting elements of both $\mc S$ or $\mc S_+$ by $h$, as the meaning should be clear from the context.

\subsection{Reminder on group cohomology}
\label{secgc}

In the next section, we discuss the differentiable structure on our spaces. We will then use basic properties of representation varieties, which we recall here. We briefly describe a construction that we will use several times. Let $\rho: \pi_1S \rar G$ be a representation into a Lie group $G$ and $\mf m$ be a $G$-module (most frequently we use $\mf m=\mf g$, though not only). It becomes a $\pi_1S$-module via the representation $\rho$. We denote by $Z^1_\rho(\mf m)=Z^1_\rho(\pi_1S, \mf m)$ the vector space of (first) $\mf m$-valued \emph{cocycles} of $\pi_1S$, i.e., maps $\tau: \pi_1S \rar \mf m$ for all $\gamma_1, \gamma_2\in \pi_1S$ satisfying
\[\tau(\gamma_1\gamma_2)=\tau(\gamma_1)+\rho(\gamma_1)\tau(\gamma_2).\]
A cocycle is called \emph{coboundary} if for some $x \in \mf m$ and all $\gamma \in \pi_1S$ we have 
\[\tau(\gamma)=\rho(\gamma)x-x.\]
Denote the subspace of coboundaries by $B^1_\rho(\mf m)=B^1_\rho(\pi_1S,\mf m)$ and define the (first) $\mf m$-valued cohomology of $\pi_1S$ as \[H^1_\rho(\mf m)=H^1(\pi_1, \mf m):=Z^1_\rho(\mf m)/B^1_\rho(\mf m).\]

In particular, if $\hat{\mc R}$ is the representation variety of $\pi_1S$ into $G$, it is convenient to perceive its Zariski tangent space at $\rho$ as $Z^1_\rho(\mf g)$, where $\mf g$ is considered as a $G$-module via the adjoint representation. This is a construction going back to Weil~\cite{Wei}. See also the book~\cite{Rag} for details. Then $B^1_\rho(\mf g)$ is the Zariski tangent space to the scheme associated to the $G$-orbit of $\rho$ under conjugation, and $H^1_\rho(\mf g)$ can be interpreted as the Zariski tangent space to the scheme associated to the quotient of $\hat{\mc R}$ by conjugation of $G$. In our cases, the components of the representation variety under our consideration are smooth, the action of $G$ on it is free and proper, and all these can be thought as the standard tangent spaces.

\subsection{Differentiable structures on spaces of embeddings}

As a main reference for Fr\'echet spaces and Fr\'echet manifolds we use the excellent survey~\cite{Ham} of Hamilton.

Consider a smooth path $(\rho_t, f_t) \in \hat{\mc E}_+$ with $(\rho_0, f_0)=(\rho, f)$. As recalled in Section~\ref{secgc}, a tangent vector to  $\rho_t: \pi_1S \rar G$ at $t=0$ can be described by a cocycle $\dot\rho \in Z^1_\rho(\mf g)$, which is a map $\dot\rho: \pi_1S \rar \mf g$. A tangent vector to $f_t$ at $t=0$ is a vector field $\dot f$ defined on $f(\tilde S)$. An easy computation shows that the condition of $\rho_t$-equivariance of $f_t$ translates at the first order into the condition of \emph{automorphicity} with respect to $\dot\rho$:
\begin{equation}
\label{automorph}
%% \dot f(\rho(\gamma)x)=d\rho(\gamma).(\dot f(x)+\dot\rho(\gamma)(x)).
\dot f (\rho(\gamma)x) = d\rho(\gamma)(\dot f (x)) + \dot \rho(\gamma)(\rho(\gamma) x)~.
\end{equation}
Here elements of $\mf g$ are interpreted as Killing fields on $\ads^3$. An automorphic vector field is \emph{trivial} if it is the restriction of a Killing field of $\ads^3$. Note that this in particular implies that $\dot\rho$ is a coboundary. 

Pick $(\rho, f) \in \hat{\mc E}_+$. Denote by $\tilde V_f$ the restriction of $T\ads^3$ to $f(\tilde S)$. Denote by $C^\infty_\rho(\tilde V_f)$ the space of vector fields that are automorphic with respect to some $\dot\rho \in Z^1_\rho(\mf g)$. This is a Fr\'echet space, and $\hat{\mc E}_+$ can be considered as a Fr\'echet manifold with charts in $C^\infty_\rho(\tilde V_f)$. In particular, $T_{(\rho, f)}\hat{\mc E}_+ \cong C^\infty_\rho(\tilde V_f)$. There is a natural fibration $\hat\pi_+: \hat{\mc E}_+ \rar \hat{\mc R}$, i.e., a surjective submersion. The group $G$ acts on $\hat{\mc R}$ freely and properly, and so it does on $\hat{\mc E}_+$. The quotient $\mc E_+$ is also naturally a Fr\'echet manifold, and there is a surjective submersion $\pi_+: \mc E_+ \rar \mc R$. The same relates to the spaces $\hat{\mc E}_-$, $\mc E_-$, $\hat{\mc E}$ and $\mc E$. It is also standard that $\mc S$ is a Fr\'echet manifold. 

\subsection{Gauss-Codazzi fields on surfaces}

Pick $s \in \mc S$. Thanks to the fundamental theorem of surfaces, the space of its isometric embeddings to $\ads^3$ has a nice parametrization in terms of \emph{GC-fields}.
%Denote by $\hat{\mc E}_+(h)$ and $\hat{\mc E}_-(h)$ the spaces of equivariant future-convex (resp. past-convex) isometric embeddings of $(\tilde S, h)$ to $\ads^3$. Denote by $\mc E_+(h)$ and $\mc E_-(h)$ their quotients by $G$. 
%Note that in this case, since the metric on $\tilde S$ is complete, the immersions are embeddings, see [Mess, or Notes on Mess]. 
%The fundamental theorem of surfaces [Ref?] allows to interpret ${\mc E}_\pm(h)$ in terms of \emph{GC-fields}.
Denote the curvature of $s$ by $\kappa: S \rar \R$ and its Levi-Civita connection by $\nabla$. A \emph{GC-field} for $s$ is an $s$-symmetric operator field $b: TS \rar TS$ satisfying the Gauss--Codazzi equations: \\
(Gauss) $\det(b)=-\kappa-1$; \\
(Codazzi) $d^\nabla b=0$. 

Denote by $\mc{GC}(s) \subset C^\infty(T^1_1S)$ the space of GC-fields on $(S, s)$. It is known from~\cite[Lemma 3.2]{Tam} of Tamburelli that it is a finite-dimensional submanifold of dimension $6k-6$. Denote by $\mc{GC}_+(s)$ and $\mc{GC}_-(s)$ the subspaces of those $GC$-fields that have positive (resp. negative) eigenvalues. The fundamental theorem of surfaces (see, e.g.,~\cite[Section 7]{BGM}) allows to define two maps
\[\psi_{s, \pm}: \mc{GC}_\pm(s) \rar \mc E_\pm~\]
so that a GC-field is the shape operator of the respective embedding. An examination of the proof of the fundamental theorem allows to ensure that 

\begin{thm}
\label{fts}
$\psi_{s, \pm}$ are $C^1$-immersions. 
\end{thm}

We will benefit from employing the dual viewpoint on the tangent vectors to $\mc {GC}_\pm(s)$: as tensor fields and as automorphic vector fields.

% ensures that $\mc{GC}_\pm(h)$ are diffeomorphic to $\mc E_\pm(h)$ respectively. (The standard formulation of the fundamental theorem of surfaces provides a bijection, but it is straightforward to check that it is smooth in both directions.) We will need to employ both these viewpoints, particularly since they give two different interpretations of tangent spaces, which both will be helpful to us. A GC-field is the shape operator of the respective embedding.

For a given $s \in \mc S$ with the almost-complex structure $j$ and a GC-field $b$, consider the following two metrics 
\begin{equation}
\label{ks}
s(({\rm id}+jb)\cdot,({\rm id}+jb)\cdot),~~~s(({\rm id}-jb)\cdot,({\rm id}-jb)\cdot)),
\end{equation}
where ${\rm id}$ is the identity operator.
It was shown by Krasnov--Schlenker~\cite{KS} that those are hyperbolic metrics on $S$. Furthermore, they represent the left and the right metrics of the corresponding quasi-Fuchsian representation. By taking the isotopy classes, we obtain two smooth maps
\[\phi_{s,\pm}: \mc{GC}_\pm(s) \rar \ms T \times \ms T\cong \mc R\]
\[b \mapsto \Big(\big[s(({\rm id}+jb)\cdot,({\rm id}+jb)\cdot)\big],\big[s(({\rm id}-jb)\cdot,({\rm id}-jb)\cdot))\big]\Big).\]
Note that $\phi_{s, \pm}$ coincide with $\pi_\pm \circ \psi_{s, \pm}$. It was shown by Chen--Schlenker~\cite{CS} that

\begin{thm}
\label{CS}
For every $s \in \mc S$ the compositions of $\phi_{s, \pm}$ with the projection to any of the factors in $\ms T \times \ms T$ is a diffeomorphism.
\end{thm}

We will need from this a weaker fact that both $\phi_{s, \pm}$ are immersions. Curiously, we do not know a proof of this not relying on Theorem~\ref{CS}.
We will also rely on the work of Labourie--Schlenker~\cite{LS}, which implies

\begin{thm}
\label{LS}
The restriction $\mc I|_{\mc E_F}$ is a homeomorphism onto the image.
\end{thm}

 %% preliminaries

\section{Infinitesimal rigidity}
\label{sc:rigidity}

This section is devoted to a proof of

\begin{prop}
\label{infrig}
Let $e \in \mc E_F$. Then $d_e\mc I$ is injective.
\end{prop}

To this goal, in the next subsection we deal first with equivariant convex surfaces in Minkowski 3-space $\R^{2,1}$. We deduce from the work of Smith~\cite{Smi} that a pair of strictly future-convex and past-convex surfaces in $\R^{2,1}$, equivariant with respect to the same quasi-Fuchsian representation, is rigid with respect to the simultaneous automorphic infinitesimal isometric deformations. After that, in the next subsection we introduce the infinitesimal Pogorelov map from anti-de Sitter 3-space to Minkowski 3-space, and conclude Proposition~\ref{infrig} from the corresponding Minkowski result.

\subsection{Equivariant surfaces in Minkowski 3-space}
\label{secmink}

Let $\R^{2,1}$ be Minkowski 3-space with the origin $o$. Denote the identity component of its isometry group by $G^0$. We consider here $G_F$ as a subgroup of $G^0$ stabilizing $o$. We have 
\begin{equation}
\label{g^0}
G^0\cong G_F \ltimes \R^{2,1}.
\end{equation}

As previously, we call a representation $\rho: \pi_1S \rar G^0$ \emph{quasi-Fuchsian} if it is a deformation of a Fuchsian representation $\pi_1S \rar G_F$. They form a component of the representation variety. We denote by $\mc R^0$ the quotient of this component by conjugation of $G^0$, while we denote the component itself by $\hat{\mc R}^0$. They and the associated spacetimes were actively studied also by Mess~\cite{Mes}, as well as by Barbot, Bonsante, Seppi and others~\cite{Bar, Bon, BBZ, BS}.

For $\rho \in \hat {\mc R}^0$, its linear part, i.e., the projection to the $G_F$-factor in~(\ref{g^0}), is a Fuchsian representation $\rho_F \in \hat{\mc R}_F$. It is twisted by a cocycle $\tau \in Z^1_{\rho_F}(\R^{2,1})$, where $\R^{2,1}$ is considered as a $G_F$-module via the standard action by isometries. This means that for $\gamma \in \pi_1S$ and $x \in \R^{2,1}$ we have
\[\rho(\gamma)x=\rho_F(\gamma)x+\tau(\gamma).\]
The space $\hat{\mc R}^0$ is a (trivial) vector bundle over $\hat{\mc R}_F$ with the fiber at $\rho_F \in \hat{\mc R}_F$ identified to $Z^1_{\rho_F}(\R^{2,1})$. The conjugacy action of $G^0$, in particular, contains an action by coboundaries on each fiber, and $\mc R^0$ is a (trivial) vector bundle over $\mc R_F$ with the fiber at $\rho_F \in {\mc R}_F$ identified to $H^1_{\rho_F}(\R^{2,1})$.

The Minkowski cross-product allows to construct a canonical isomorphism between $\R^{2,1}$ and $\mf g_F$ as $G_F$-modules. See, e.g.,~\cite[Section 3]{FS} of Fillastre--Seppi. Thereby, we have identification $H^1_{\rho_F}(\mf g_F)\cong H^1_{\rho_F}(\R^{2,1})$. Furthermore, we have 
\begin{equation}
\label{rep0}
\mc R^0\cong T\mc R_F \cong T \ms T.
\end{equation}

We can repeat some definitions from Section~\ref{secembed} for immersions of surfaces into Minkowski 3-space.
Let $s$ be a negatively curved metric on $S$ and $\kappa: S \rar \R$ be its curvature. Denote by $\mc{GC}^0_+(s)$ and $\mc{GC}^0_-(s)$ the spaces of GC-fields on $(S, s)$ that have both eigenvalues of the respective sign. (For Minkowski 3-space the Gauss equation is $\det(b)=-\kappa$.) Similarly to the anti-de Sitter situation, they are endowed with topologies of smooth manifolds of dimension $6k-6$. The fundamental theorem of surfaces allows to interpret each element of $\mc{GC}^0_\pm(s)$ as a spacelike equivariant surface in $\R^{2,1}$. By considering the respective representation, we obtain two smooth maps 
\[\phi_{s,\pm}^0: \mc{GC}^0_\pm(s) \rar T\ms T.\]
Denote by $\chi_{s,\pm}^0: \mc{GC}^0_\pm(s) \rar \ms T$ the compositions of $\phi_{s,\pm}^0$ with the projection to $\ms T$.
For $b \in \mc{GC}^0_\pm(s)$, the element $\chi_{s, \pm}^0(b)$ is given by the isotopy class of the metric $s(b\cdot, b\cdot)$, see, e.g.,~\cite{Smi}.

%Denote by $\hat{\mc E}_+^0(h)$ and $\hat{\mc E}_-^0(h)$ the spaces of future-convex (resp. past-convex) equivariant isometric immersions of $(S, h)$ to $\R^{2,1}$. Once again, the work of Mess~\cite{Mes} implies that the representations are quasi-Fuchsian and the immersions are embeddings. Denote by $\mc E_+^0(h)$ and $\mc E_-^0(h)$ the $G^0$-quotients of $\hat{\mc E}_+^0(h)$ and $\hat{\mc E}_-^0(h)$. We have maps $\hat {\mc E}^0_\pm(h) \rar \hat{\mc R}^0$, descending to maps $\mc E^0_\pm(h) \rar \mc R^0$, which we interpret with a help of~(\ref{rep0}) as two smooth maps
%\[\phi_{h,\pm}^0: \mc E^0_\pm(h) \rar T\ms T.\]

Note that in contrast to the anti-de Sitter situation, a pair of a future-convex and a past-convex surfaces in Minkowski 3-space, equivariant with respect to the same representation, never bound together a subset, on which the action of the representation is free and properly discontinuous. A quasi-Fuchsian representation into $G^0$ has two disconnected maximal convex domains of discontinuity in $\R^{2,1}$, one is future-convex and one is past-convex.

The following result is a corollary of the work of Trapani--Valli~\cite{TV}, as was noticed by Smith in~\cite{Smi}.

\begin{thm}
\label{TV}
For any negatively curved metric $s$, $\chi_{s, \pm}^0$ are diffeomorphisms.
\end{thm}

In particular, the maps $\phi^0_{s, \pm}$ are immersions. For $\sigma \in \ms T$ define 
\[x_s(\sigma):=\phi_{s,+}^0\circ (\chi_{s, +}^0)^{-1}(\sigma).\]
This gives us a vector field $x_s$ on $\ms T$ defined by $s$. If $f: \tilde S \rar \R^{2,1}$ is a $\rho$-equivariant future-convex isometric embedding for a quasi-Fuchsian $\rho$ with cocycle $\tau$, then $-f$ is a past-convex isometric embedding equivariant with respect to a quasi-Fuchsian representation with the same linear part and $-\tau$ as the cocycle. Hence, 
\[-x_s(\sigma)=\phi_{s,-}^0\circ (\chi_{s,-}^0)^{-1}(\sigma).\]

There is another related vector field over $\ms T$ that we associate to $s$. 
%Consider an equivariant isometric embedding $(\rho, f): \tilde S \rar \R^{2,1}$ with the induced metric $s$. Let $\rho_F$ be the linear part of $\rho$ and $\sigma$ be the hyperbolic metric on $S$ given by $\H^2 / \rho_F$ and pulled to $S$ by the composition of the Gauss map from $f(\tilde S)$ and of $f$. Let $b$ be the shape operator of $f(\tilde S)$, pulled to $S$, and $a:=b^{-1}$. Note that $s=\sigma(a\cdot, a\cdot)$ and $a$ is Codazzi for $\sigma$, see, e.g.,~\cite[Proposition 2.7]{BS} of Bonsante--Seppi.
To proceed, we need to recall that due to the uniformization theorem, the Teichm\"uller space $\ms T$ can be also considered as the space of isotopy classes of conformal structures. See, e.g.,~\cite{Tro} for details. 
Let $\mc S^*$ be the space of all Riemannian metrics on $S$. For $\sigma \in \mc S^*$ we denote by $\mc O_\sigma$ the space of $\sigma$-symmetric operator fields on $S$. Any $\sigma \in \mc S^*$ determines an isomorphism between the subspace of positive-definite elements of $\mc O_\sigma$ and $\mc S^*$ by $a \mapsto \sigma(a.,.)$. Denote the differential at the identity of this map by $\zeta_\sigma:\mc O_\sigma \rar T_\sigma \mc S^*$, which is also an isomorphism. By taking the conformal class of a metric, we get the projection $\mc S^* \rar \ms T$. We compose the differential of this projection at $\sigma$ with $\zeta_\sigma$ and obtain a map
\[\eta_\sigma: \mc O_\sigma \rar T_\sigma\ms T.\]
Now define another vector field $x_s'$ on $\ms T$ by
\[x_s'(\sigma):=\eta_\sigma \circ {\rm inv} \circ (\chi_{s,+}^0)^{-1}(\sigma),\]
where ${\rm inv}$ is the operation of taking the inverse on $C^\infty(T^1_1 S)$, i.e., ${\rm inv}(b)=b^{-1}$.

Recall that there is a canonical complex structure on $\ms T$, which we denote by $\ms J$. Bonsante--Seppi showed that $\ms Jx'_s=x_s$ \cite[Theorem B]{BS}.

%\hrulefill
%
%A Riemannian metric $\sigma$ determines an isomorphism between the space of $\sigma$-symmetric positive-definite operator fields and the space of Riemannian metrics by $a \mapsto \sigma(a.,.)$. This allows to interpret the space of $\sigma$-symmetric operator fields as the tangent space at $\sigma$ to the space of Riemannian metrics. By taking the conformal class of a metric, we see that such an operator gives a tangent vector to $\ms T$ at the class of $\sigma$. We apply this to $\sigma$ and $a$ above. The result depends only on the isometry class of $(\rho, f)$, i.e., is determined by the shape operator $b \in \mc{GC}^0_+(s)$. Together with Theorem~\ref{TV}, this construction produces another vector field on $\ms T$, which we denote by $x'_s$. Recall that there is a canonical complex structure on $\ms T$, which we denote by $\ms J$. Bonsante--Seppi showed that $\ms Jx'_s=x_s$ \cite[Theorem B]{BS}.
%
%\hrulefill

A Riemannian metric $\sigma$ on $S$ determines an $L^2$-scalar product on all tensor fields over $S$. For operator fields $a_1, a_2$, which are $(1,1)$-tensor fields, it is expressed as
\[\la a_1,a_2 \ra_{L^2}=\int_S {\rm tr}(a_1a_2)d{\rm area}_\sigma.\]
The element $\zeta_\sigma(a)$ is orthogonal to the tangent space to the isotopy orbit of $\sigma$ if and only if $a$ is Codazzi for $\sigma$. See, e.g.,~\cite[Section 3.2]{Yam}. It is also easy to notice that $\zeta_\sigma(a)$ is orthogonal to the tangent space to the conformal orbit of $\sigma$ if and only if $a$ is traceless. We arrive to the well-known fact that that the restriction of $\eta_\sigma$ to the subspace of $\sigma$-symmetric traceless Codazzi operator fields is an isomorphism. When $\sigma$ is the hyperbolic representative of the conformal class, the $L^2$-scalar product produces the Weil--Petersson scalar product on the Teichm\"uller space, up to a universal constant depending on a convention. It also follows from the discussion that if $x_1, x_2 \in T_{\sigma}\ms T$ and $a_1 \in \eta_\sigma^{-1}(x_1)$, $a_2\in \eta_\sigma^{-1}(x_2)$ are such that $a_1$ and $a_2$ are Codazzi and $a_2$ is traceless, then
\[\la x_1, x_2 \ra_{WP}=\int_S{\rm tr}(a_1a_2)d{\rm area}_\sigma.\]

Now return to an equivariant isometric immersion $(\rho, f)$ with induced metric $s$ and shape operator $b$. Let $\rho_F$ be the linear part of $\rho$ and $\sigma$ be the hyperbolic metric on $S$ given by $\H^2 / \rho_F$ and pulled to $S$ by the composition of the Gauss map from $f(\tilde S)$ and of $f$. Define $a:=b^{-1}$, i.e., $a={\rm inv}\circ(\chi_{s,+}^0)^{-1}(\sigma)$. Then $s=\sigma(a\cdot, a\cdot)$ and $a$ is Codazzi for $\sigma$, see, e.g.,~\cite[Proposition 2.7]{BS} of Bonsante--Seppi. Define
\[F(s, \sigma):=\int_S {\rm tr}(a)d{\rm area}_{\sigma}.\]
Due to Theorem~\ref{TV}, $F$ is indeed a well-defined function of a negatively curved metric $s$ and a hyperbolic metric $\sigma$. Furthermore, it depends only on the isotopy classes of $s$, $\sigma$. For fixed $s$, consider $F(s, \cdot)$ as a function $F_s$ over the Teichm\"uller space $\ms T$. It was shown by Smith in~\cite[Lemma 5.3]{Smi} that

\begin{lm}
Pick $x\in T_{\sigma}\ms T$. If $a'\in \eta_\sigma^{-1}(x)$ is traceless and Codazzi, then
\[\dot F_s=-\int_S {\rm tr}(aa')d\area_{\sigma}.\] 
\end{lm}

\begin{crl}
The vector field $-x'_s$ is the Weil--Petersson gradient field of $F_s$ and $-x_s$ is the symplectic Weil--Petersson gradient field of $F_s$.
\end{crl}

Furthermore, Smith proved~\cite[Lemma 5.6]{Smi}

\begin{lm}
$F_s$ is strictly convex with respect to the Weil--Petersson metric.
\end{lm}

From this we can deduce

\begin{crl}
\label{transverse}
For any two negatively curved metrics $s_+$ and $s_-$, the vector fields $x_{s_+}'$ and $-x_{s_-}'$ intersect transversely at every intersection point as submanifolds of $T\ms T$. The same holds for $x_{s_+}$ and $-x_{s_-}$.
\end{crl}

\begin{proof}
If at a point $\sigma \in \ms T$ we have $x_{s_+}'(\sigma)=-x_{s_-}'(\sigma)$, then the Weil--Petersson gradients of functions $F_{s_+}$ and $-F_{s_-}$ coincide at $\sigma$. Thus, $\sigma$ is a critical point of $F_{s_+}+F_{s_-}$. Since the latter function is strictly convex with respect to the Weil--Petersson metric, its Weil--Petersson Hessian is non-degenerate. Thereby, the intersection point of $x_{s_+}'$ and $-x_{s_-}'$ is transverse at $x_{s_+}'(\sigma)=-x_{s_-}'(\sigma)$. The same holds for the vector fields $x_{s_+}$, $-x_{s_-}$.
\end{proof}

Note that Smith also proved in~\cite{Smi} that $F_s$ is proper. Thereby, an intersection point in Corollary~\ref{transverse} exists for any pair of metrics. Since $\ms T$ is geodesically convex with respect to the Weil--Petersson metric, it is unique. However, we will not rely on this.
All the discussion allows us to obtain

\begin{lm}
\label{infrigmink}
Let $\rho: \pi_1 S \rar G^0$ be a quasi-Fuchsian representation and $f_+, f_-: \tilde S \rar \R^{2,1}$ be two $\rho$-equivariant strictly convex embeddings, one is future-convex and one is past-convex. Let $\dot \rho \in Z^1_\rho(\mf g^0)$ be a cocycle and $v$ be a $\dot\rho$-automorphic isometric vector field on $f_+(\tilde S)$ and $f_-(\tilde S)$. Then $v$ is trivial, i.e., its restrictions to both $f_+(\tilde S)$ and $f_-(\tilde S)$ coincide with the restrictions of the same global Killing field of $\R^{2,1}$.
\end{lm} 

Here $v$ is an isometric vector field on $f_\pm(\tilde S)$ if $(\mc L_v g)(u_1, u_2)=0$ for all tangent vectors $u_1$, $u_2$ to $f_\pm(\tilde S)$, where $g$ is the Minkowski metric. 

\begin{proof}
Let $s_+$ and $s_-$ be the induced metrics, $b_+$ and $b_-$ be the shape operators of $f_+$ and $f_-$ and $\dot b_+$ and $\dot b_-$ be their variations induced by $v$. If $\dot b_+=0$, then, by Theorem~\ref{fts}, $v$ is trivial on $f_+(\tilde S)$. Hence $\dot\rho$ is a coboundary. Then, due to Theorem~\ref{TV}, also $\dot b_-=0$ and hence, by Theorem~\ref{fts}, $v$ is trivial also on $f_-(\tilde S)$. Furthermore, on both surfaces it is the restriction of the same Killing field, determined by $\dot\rho$. Thus, we can assume that both $\dot b_+$ and $\dot b_-$ are nonzero. By definition, we have $\phi^0_{s_+,+}(b_+)=\phi^0_{s_-,-}(b_-)$. Since the restrictions of $v$ to both $f_+(\tilde S)$ and $f_-(\tilde S)$ are $\dot\rho$-automorphic, by construction, 
\begin{equation}
\label{t}
d_{b_+}\phi^0_{s_+, +}(\dot b_+)=d_{b_-}\phi^0_{s_-, -}(\dot b_-).
\end{equation}
Due to Theorem~\ref{TV}, expression~(\ref{t}) is nonzero. But this contradicts Corollary~\ref{transverse}, as the images of $\phi_{s_+, +}$ and $\phi_{s_-, -}$ are exactly the vector fields $x_{s_+}$ and $-x_{s_-}$, when the latter are considered as submanifolds of $T\ms T$.	
%The vector field $v$ represents tangent vectors $\hat x_+ \in T_{(\rho, f_+)}\hat{\mc E}^0_+$ and $\hat x_- \in T_{(\rho, f_-)}\hat{\mc E}^0_-$. Let $e_+$ and $e_-$ be the projections of $(\rho, f_+)$ and $(\rho, f_-)$ to $\mc E^0_+$ and $\mc E^0_-$. If $v$ is nontrivial, then we have tangent vectors $x_+ \in T_{e_+}\mc E^0_+$ and $x_- \in T_{e_-}\mc E^0_-$. Let $h_+$ and $h_-$ be the metrics on $S$ induced by $f_+$ and $f_-$. By definition, we have $\phi^0_{h_+,+}(e_+)=\phi^0_{h_-,-}(e_-)$. Since the restrictions of $v$ to both $f_+(\tilde S)$ and $f_-(\tilde S)$ are $\dot\rho$-automorphic, by construction, 
%\begin{equation}
%\label{t}
%d_{e_+}\phi^0_{h_+, +}(x_+)=d_{e_-}\phi^0_{h_-, -}(x_-).
%\end{equation}
%Due to Theorem~\ref{TV}, expression~(\ref{t}) is nonzero. But this contradicts to Corollary~\ref{transverse}, as the images of $\phi_{h_+, +}$ and $\phi_{h_-, -}$ are exactly the vector fields $x_{h_+}$ and $-x_{h_-}$, when the latter are considered as submanifolds of $T\ms T$.
\end{proof}

\subsection{The infinitesimal Pogorelov map}
\label{ssc:pogorelov}

Recall the definition of $\ads^3$ from Section~\ref{secads}. We now pick a convenient affine chart for it, given by the projection to the plane $\{x \in \R^{2,2}: x_4=1\}$. This plane is also a model for Minkowski 3-space, and we will consider it endowed simultaneously with both geometries. We denote it now by $\R^{2,1}$, and by the origin we will mean the point $o=(0,0,0,1)$. The part of $\ads^3$ that is captured in this plane is the interior of a one-sheeted hyperboloid
\[A:=\{x \in \R^{2,1}: x_1^2+x_2^2-x_3^2 < 1\}.\]
Let $C$ be the union of the future- and the past-cones of $o$. (Note that the future- and the past-cones of $o$ coincide both in Minkowski and in anti-de Sitter geometries.) Denote by $r: C \rar \R_{>0}$ the Minkowski distance to $o$ on $C$, considered as a positive real number.

%Consider $\R^3$ with the coordinates $(x_0, x_1, x_2)$ and with the origin $o$. We denote by $\mc A \subset \R^3$ the set $\{x \in \R^3: x_1^2+x_2^2-x_0^2<1\}$ and denote its boundary by $\pt_\infty \mc A$, which is a one-sheeted hyperboloid. We consider $\mc A$ equipped simultaneously with AdS and with Minkowski metrics. We denote by $r_A$ and $r_M$ the respective distances to $o$. We have $r_M=\tanh(r_A)$. We denote by $R_A$ and $R_M$ the unit radial vector fields in the respective geometries. Note that the orthogonality to the radial directions is the same in the AdS and in the Minkowski metrics. We also remark that the Minkowski metric is obtained from the AdS metric by multiplying by $(1-r_M^2)^2$ in the radial directions and by $1-r_M^2$ in the lateral directions, where by lateral we mean orthogonal to radial.

There is an (infinitesimal) Pogorelov map $\Phi:T C \rar T C$ defined as follows. For $x \in C$ and $v \in T_x C$ let $v_r$ be its radial component, i.e., parallel to the direction of the ray $ox$, and $v_l$ be its lateral component, i.e., orthogonal to the radial direction. (Note that the orthogonality to the radial directions is the same both in Minkowski and in anti-de Sitter geometries.) Then 
\[\Phi(v):=\frac{v_r}{1-r^2(x)}+v_l.\]
The map $\Phi$ is an automorphism of $T C$. Its key property is

\begin{lm}
\label{ipm}
The map $\Phi$ sends anti-de Sitter Killing fields to Minkowski Killing fields. The inverse $\Phi^{-1}$ sends Minkowski Killing fields to anti-de Sitter Killing fields.
\end{lm}

Here, for convenience, we only consider the restrictions of the Killing fields to $C$, though one can define $\Phi$ over the whole tangent bundle $TA$ of $A$, see, e.g.,~\cite{DMS}.
See~\cite[Lemma 11]{Fil} of Fillastre for a proof of Lemma~\ref{ipm}, which is based on~\cite[Lemma 1.10]{Sch} of Schlenker. In particular, $\Phi$ determines an isomorphism $\Psi: \mf g \rar \mf g^0$ of Lie algebras, considered as vector spaces. Furthermore, every $\gamma \in G_F$ preserves the decomposition into radial and lateral parts, as well as the function $r$. Hence, $\Phi$ is $G_F$-equivariant. In turn, this implies

\begin{lm}
\label{modules}
$\Psi$ is an isomorphism of $G_F$-modules.
\end{lm}

%\begin{proof}
%We have decompositions $\mf g=\mf g_r \oplus \mf g_t$, $\mf g^0=\mf g_{r}^0 \oplus \mf g_{t}^0$ into pure rotations and pure translations with respect to $o$. Here also $\mf g_r\cong\mf g_r^0\cong \mf g_F$ via the chosen inclusions of $G_F$ to $G$ and $G^0$. Killing fields from the first summands are characterized by the property that they do not have radial component. Killing fields from the second summands are characterized by the property that along some geodesic passing through $o$ they do not have a lateral component. (Here the geodesic depends on a field.) It is clear from this description that both $G_F$ and $\Psi$ preserve these decompositions. Furthermore, $\Psi$ preserves the isomorphisms $\mf g_r\cong\mf g_r^0\cong \mf g_F$.
%
%The $G_F$-action on the first components coincides with the $G_F$-action on $\mf g_F$ and is preserved by $\Psi$. It is an easy observation that for $\gamma \in G_F$ and $\xi \in \mf g_t$ that is an infinitesimal translation along a line $l$ passing through $o$, $\gamma \xi$ is an infinitesimal translation along the line $\gamma l$ with the same norm. It follows that $\Psi$ preserves the $G_F$-actions on $\mf g$ and $\mf g^0$.
%\end{proof}

For a Fuchsian $\rho: \pi_1S \rar G_F$, Lemma~\ref{modules} establishes a canonical isomorphism \[\Psi_\rho: Z^1_\rho(\mf g) \rar Z^1_\rho(\mf g^0),\] which sends coboundaries onto coboundaries, and hence also induces an isomorphism of cohomologies. Now we also note

\begin{lm}
\label{pogisom}
Let $\Sigma \subset C$ be a submanifold and $v$ be an infinitesimal isometric deformation of $\Sigma$ with respect to the anti-de Sitter metric. Then $\Phi(v)$ is an infinitesimal isometric deformation of $\Sigma$ with respect to the Minkowski metric.
\end{lm}

A proof follows word by word the proof of Lemma 1.10 in~\cite{Sch}, which deals with the Pogorelov map from hyperbolic space to Euclidean space. The role of relevant metric computations is played by Lemmas 12 and 13 from~\cite{Fil}.

\begin{lm}
\label{pogaut}
Let $\rho \in \hat{\mc R}_F$, $\Sigma \subset C$ be a $\rho$-invariant submanifold and $v$ be a vector field over $\Sigma$, automorphic with respect to $\dot\rho \in Z^1_\rho(\mf g)$. Then $\Phi(v)$ is automorphic with respect to $\Psi_\rho(\dot \rho) \in Z^1_\rho(\mf g^0)$.
\end{lm}

\begin{proof}
For all vector fields over $\Sigma$ we use subscripts $r$, $l$ to denote the radial and the lateral components. The key point is that Fuchsian representations preserve the function $r$, and hence the Pogorelov map is well-behaved. Here is the computation. For $x \in \Sigma$ we have
\[d\rho(\gamma)\big(\Phi(v(x))\big)+\Psi_\rho(\dot\rho)(\gamma)(\rho(\gamma)x)=\frac{d\rho(\gamma)(v_r(x))+\dot\rho(\gamma)_r(\rho(\gamma)x)}{1-r^2(x)}+\]
\[+d\rho(\gamma)(v_l(x))+\dot\rho(\gamma)_l(\rho(\gamma)x)
=\frac{v_r(\rho(\gamma)x)}{1-r^2(\rho(\gamma)x)}+v_l(\rho(\gamma)x)=\Phi(v(\rho(\gamma)x)),\]
which verifies the automorphicity condition~(\ref{automorph}).
\end{proof}

Now we have all in our hands to conclude a proof.

\begin{proof}[Proof of Proposition~\ref{infrig}.]
Consider $e \in \mc E_F$. Lift it to a triple $(\rho, f_+, f_-) \in \hat{\mc E}_F$ of a representation $\rho: \pi_1 S \rar G_F$, where $G_F$ is considered as a subgroup of $G$, and of $\rho$-equivariant developing maps $f_+, f_-: \tilde S \rar \ads^3$. We assume that $o$ is the fixed point of $G_F$ as a subgroup of $G$. Then the images of $f_+$ and $f_-$ are contained in $C$. Pick $\dot e \in T_e \mc E$ such that $d_e\mc I(\dot e)=0$. It produces a cocycle $\dot\rho \in Z^1_\rho(\mf g)$ and a $\dot\rho$-automorphic vector field $v$ on $f_+(\tilde S)$ and $f_-(\tilde S)$, isometric with respect to the anti-de Sitter metric.

Apply the Pogorelov map $\Phi$ to $v$. The combination of Lemma~\ref{pogaut} and Lemma~\ref{pogisom} implies that $\Phi(v)$ is a $\Psi_\rho(\dot \rho)$-automorphic vector field on $f_+(\tilde S)$ and $f_-(\tilde S)$, isometric with respect to the Minkowski metric. Hence, by Lemma~\ref{infrigmink}, $\Phi(v)$ is trivial, thereby so is $v$. It follows that $\dot e=0$.
\end{proof}

\begin{rmk}
It is interesting to notice that even though one can define the infinitesimal Pogorelov map outside of $C$, this proof works only when the pair of embeddings is Fuchsian. The reason is that for a general equivariant embedding, its image in the projective model will not be equivariant with respect to some Minkowski representation. There is a nice way to do a rescaling, to obtain quasi-Fuchsian Minkowski spacetimes from anti-de Sitter ones (more precisely, from the components of the complements to the convex core), well-behaved on holonomies, see~\cite{BB} by Benedetti--Bonsante. However, this is not a projective map, and hence does not have a related infinitesimal Pogorelov map. (An infinitesimal Pogorelov map can be associated to any projective map between ambient spaces, see~\cite{FS5}).
\end{rmk}
 %% inf rigidity

\section{Rigidity near the Fuchsian locus}

\subsection{A bit of functional analysis}
\label{secfa}

For a bounded operator $\mc L$ between two Banach spaces $\mc A, \mc B$, we define its reduced minimum modulus
\[\gamma(\mc L):=\inf\{\|\mc La\|: a\in \mc A, d(a, \ker \mc L)=1\}.\]
When the range of $\mc L$ is closed, $\gamma(\mc L)>0$. See, e.g.,~\cite[Chapter IV, Theorem 5.2]{Kat}.

For two subspaces $A,B$ of a Banach space, we define their distance as
\begin{equation}
\label{subspdist}
d(A, B):=\sup\{d(a, B): a \in A, \|a\|\leq 1\}.
\end{equation}
%For subspaces of $C^\infty(V)$ we define it the same, considering them as subspaces of $L^2(V)$.

We will need the following 

\begin{lm}
\label{subconverge}
Let $A$ be a finite-dimensional subspace in a Banach space, and $A_i$ be a sequence of subspaces, converging to $A$ in the distance~(\ref{subspdist}). Let $a_i \in A_i$ be a uniformly bounded sequence of vectors. Then, up to subsequence, $a_i$ converges to $a \in A$.
\end{lm}

\begin{proof}
Since $A$ is finite-dimensional, for every $a_i$ there is $a'_i \in A$ realizing the distance from $a_i$ to $A$. The sequence $a'_i$ is uniformly bounded. Since bounded subsets of $A$ are precompact, a subsequence of $a'_i$ converges to some $a$. It is easy to see that the same subsequence of $a_i$ also converges to $a$.
\end{proof}

\subsection{Local infinitesimal rigidity}
\label{ssc:locinf}

We show now the following result.

\begin{prop}
\label{inflocrig}
Let $e \in \mc E_F$. There exists a neighborhood $U_{e}$ of $e$ in $\mc E$ such that for every $e' \in U_{e}$ the differential $d_{e'} \mc I$ is injective.
\end{prop}

The first point is that the infinitesimal rigidity can be expressed in terms of transversality of submanifolds. (We already used this idea in the proof of Lemma~\ref{infrigmink}.)

\begin{lm}
\label{transverseloc}
Consider $e' \in \mc E$. Let $s_+$, $s_-$, $b_+$ and $b_-$ be the respective induced metrics and shape operators on $S$ given by $e'$. The differential $d_{e'}\mc I$ is injective if and only if the images of $\phi_{s_+,+}$ and $\phi_{s_-,-}$ intersect transversely at $\phi_{s_+,+}(b_+)=\phi_{s_-,-}(b_-)$.
\end{lm}

\begin{proof}
Fix a lift $(\rho, f_+, f_-) \in \hat{\mc E}$ of $e'$. Pick $\dot e \in T_{e'}\mc E$. If $d_{e'}\mc I(\dot e)=0$, then $\dot e$ can be represented by $(\dot\rho, \dot f_+, \dot f_-)$ such that $\dot f_+$ and $\dot f_-$ are isometric on $f_+(\tilde S)$ and $f_-(\tilde S)$. Suppose that $\dot f_+$ induces zero variation on the shape operator of $f_+(\tilde S)$. By Theorem~\ref{fts}, $\dot f_+$ is trivial. Then $\dot\rho$ is a coboundary and $\dot e=0$. Suppose now that it is not the case and $\dot b_+$ and $\dot b_-$ are nonzero variations of the shape operators. Then, by construction, $d_{b_+}\phi_{s_+, +}(\dot b_+)=d_{b_-}\phi_{s_-, -}(\dot b_-)$, which is nonzero due to Theorem~\ref{CS}. Thus $\phi_{s_+,+}$ and $\phi_{s_-,-}$ do not intersect transversely. The converse is shown in the same way.
\end{proof}

Now we need to show that if $s_i$ converge to $s$ and 
$\rho_i \in {\rm Im}(\phi_{s_i, \pm})$ converge to $\rho \in {\rm Im}(\phi_{s, \pm})$, then the tangent space to ${\rm Im}(\phi_{s_i, \pm})$ at $\rho_i$ converge to the tangent space to ${\rm Im}(\phi_{s, \pm})$ at $\rho$. 
Let $\mc A$ be the Sobolev space $W^{1,2}(T^1_1S)$ of sections of $T^1_1S$ that have finite $L^2$-norm and so do their first weak derivatives. We consider $\mc A$ as a Banach space.
%To this purpose we employ the interpretation of $\mc E_\pm(h)$ as GC-fields. Recall that for a metric $h$ we denote by $\mc{GC}_\pm(h) \subset C^\infty(T^1_1S)$ the spaces of GC-fields on $(S, h)$ with positive/negative eigenvalues and that $\mc{GC}_\pm(h)$ are canonically diffeomorphic to $\mc E_\pm(h)$. 
The desired convergence result follows from

\begin{lm}
\label{tangspaces}
Consider a sequence $s_i$ in $\mc S$ that $C^\infty$-converges to $s \in \mc S$. Let $b_i$ be a sequence of GC-fields for $s_i$ that $C^\infty$-converges to a GC-field $b$ for $s$. Then $T_{b_i}\mc{GC}(s_i)$ converges to $T_b\mc{GC}(s)$ as subspaces of $\mc A$ in the sense of distance~(\ref{subspdist}).
\end{lm}

\begin{proof}
For every $i$ denote by $G_i \subset T^1_1S$ the subbundle of $s_i$-symmetric operator fields $TS \rar TS$ satisfying the Gauss equation for $s_i$. Denote by $J \subset T^1_1S$ the subbundle of almost-complex structures on $S$ (compatible with the orientation) and let $j_i$, $j$ be the almost-complex structures of $s_i$, $s$. Consider the bundle isomorphisms $K_i: T^1_1S \rar T^1_1 S$ given by
\[ K_i: a \mapsto -\lambda_ij_ia, \]
where $\lambda_i:=\sqrt{-\kappa_i-1}$ for the curvature $\kappa_i$ of $s_i$.
Notice that $K_i$ induce isomorphisms between the bundles $J$ and $G_i$. This construction goes back to Labourie~\cite{Lab}, who studied immersions of surfaces into hyperbolic 3-manifolds.
%Let $\mc A$ be the Sobolev space $W^{1,2}(T^1_1S)$, considered as a Banach space. 
The maps $K_i$ determine the isomorphisms $\mc K_i: \mc A \rar \mc A$. (By isomorphisms between Banach spaces we mean bounded linear bijections with bounded inverse.) The metric $s$ similarly defines the maps $K$ and $\mc K$. The operators $\mc K_i$ converge to $\mc K$ in the sense of operator norms. Define $a_i:=\mc K_i^{-1}b_i$, $a:=\mc K^{-1}b$. Then $a_i$ $C^\infty$-converge to $a$.

The tangent space to $C^\infty(J)$ at arbitrary $j' \in C^\infty(J)$ is defined as the space of smooth operator fields $z': TS \rar TS$ satisfying $z'j'+j'z'=0$. Denote the respective bundle by $Z'$.
Let $\nabla_i$ be the Levi-Civita connections of $s_i$. The pullback of the linearization of $d^{\nabla_i}$ to $C^\infty(Z')$ by $K_i$ is given by
\[\mc L_i(z')=-\lambda_ij_i(d^{\nabla_i}z')-(d\lambda_i)j_iz'.\]
Similarly we define a differential operator $\mc L$. Tamburelli showed that for all such $Z'$ the operators $\mc L_i$, $\mc L$, as operators from $C^\infty(Z')$ to $C^\infty(\Lambda^2 TS \otimes TS)$, are elliptic of index $6k-6$ with trivial cokernel, where $k$ is the Euler characteristic of $S$. See the proof of Lemma 3.2 in~\cite{Tam}. (Compare also with \cite[Lemma 3.1]{Lab} of Labourie.) Hence every $\mc L_i$ on $C^\infty(Z_i)$ has kernel of dimension $6k-6$, where $C^\infty(Z_i)$ is the tangent space to $C^\infty(J)$ at $a_i$. We denote this kernel by $X_i$. The same holds for the kernel $X$ of $\mc L$ at $C^\infty(Z)$, the tangent space at $a$.

We claim that it is enough to show that $X_i$ converge to $X$ as subspaces of $\mc A$. Indeed, GC-fields for $s_i$ are those $b \in C^\infty(G_i)$ that satisfy $d^{\nabla_i}b=0$. Hence, $T_{b_i}\mc{GC}(s_i)=\mc K_i(X_i)$, as well as $T_b\mc{GC}(s)=\mc K(X)$. Since $\mc K_i$ are isomorphisms converging to $\mc K$, if $X_i$ converge to $X$, then $T_{b_i}\mc{GC}(s_i)$ converge to $T_b\mc{GC}(s)$. 
%We also note that we will show the convergence for the distance generated by the $W^{1,2}$-norm, which obviously implies the convergence for the distance generated by the $L^2$-norm. (We need to employ spaces of $W^{1,2}$-sections since our operators extend there as bounded operators.) 
To show the convergence of $X_i$, we construct a sequence of bundle automorphisms of $T^1_1 S$ converging to the identity that send $Z_i$ to $Z$. To this purpose, we show that there exists a sequence of projectivized operator fields $p_i \in C^\infty(\mathbb P(T^1_1S))$, $C^\infty$-converging to the identity, such that $p_ia_ip_i^{-1}=a$. Here we note that the conjugation operation is well-defined for invertible elements of the projectivized operator bundle.

For the moment, let us consider a vector space $V\cong \R^2$ with some orientation. Let $J$ be the space of orientation-preserving almost-complex structures on $V$. Pick $a, a_1 \in J$, $a \neq a_1$. We claim that there exists a canonical $p \in {\rm PSL}(2, \R)$ such that $pa_1p^{-1}=a$, where by canonical we mean that it is independent on a choice of basis or on a choice of any additional structure. Indeed, under the identification ${\rm PSL}(2, \R) \cong \ads^3$ the space $J$ becomes the dual plane to the identity, or, in other words, the set of points at timelike distance $\pi/2$ from the identity. Under the identification \[{\rm PSL}(2, \R) \times {\rm PSL}(2, \R) \cong G\] the diagonal is the set of isometries fixing the identity.
Hence, there exists a unique isometry fixing the identity, sending $a_1$ to $a$ and preserving the geodesic between them. This isometry is given by conjugation by $p \in {\rm PSL}(2, \R)$.

Now we return to our setting. For every $i$ at every point of $S$ we pick the respective element of the fiber of $\mathbb P(T^1_1S)$ conjugating $a_i$ to $a$. This produces a smooth section $p_i$ of the projectivized bundle $\mathbb P(T^1_1 S)$ such that $p_ia_ip_i^{-1}=a$. 

%Since $S$ is orientable, the first Stiefel--Witney class of $TS$ vanishes, hence so does the first Stiefel--Witney class of $T^1_1S$, hence $\tilde a_i$ lifts to a section $a_i$ of $T^1_1S$. [Note that we actually do not need to lift, because what we need is just a smooth automorphism of $C^\infty(T^1_1S)$ preserving $J$ and sending $j_i$ to $j$, which is actually given by the projectivized section.]

Denote now $W^{1,2}(Z)$ by $\mc Z$. The conjugation by the projectivized operator fields $p_i$ define isomorphisms $\mc R_i: \mc A \rar \mc A$ that converge to the identity in the sense of operator norms. The operators $(\mc R_i)_*\mc L_i$ are Fredholm operators from $\mc Z$ to $L^2(\Lambda^2 TS \otimes TS)$ that converge to a Fredholm operator $\mc L$.  All $(\mc R_i)_*\mc L_i$ and $\mc L$ have the same dimension of kernel and zero cokernel. Denote the kernels of $(\mc R_i)_*\mc L_i$ by $Y_i$. For every $a' \in \mc A$ we have
\[ d(a', X) \gamma(\mc L) \leq \|\mc L a'\|.\]
Since $\mc L$ is Fredholm, its range is closed, thus $\gamma(\mc L)>0$. It follows that for $a' \in Y_i$, $\|a'\|\leq 1$, we get 
\[ d(a', X)  \leq \gamma^{-1}(\mc L)\|(\mc L-(\mc R_i)_*\mc L_i)\|.\]
Thus, $Y_i$ converge to $X$. In turn, $\mc R^{-1}(Y_i)=X_i$ also converge to $X$. 
%Finally, the morphisms $F_i$ define a converging sequence of bounded operators $\psi_i: L^2(J) \rar \mc D$. It follows that $X_i:=\psi_i(Y_i)$ converge to $X_i:=\psi(Y)$.
\end{proof}

We can now prove Proposition~\ref{inflocrig}.

\begin{proof}[Proof of Proposition~\ref{inflocrig}.]
Suppose the converse. Then there exists a sequence $e_i \in \mc E$ converging to $e$ such that the differentials $d_{e_i}\mc I$ are not injective. Let $(s_{i,+}, s_{i, -})$ and $(b_{i,+}, b_{i,-})$ be the corresponding induced metrics and shape operators on $S$. Similarly, let $(s_+, s_-)$ and $(b_+, b_-)$ be induced by $e$. By Lemma~\ref{transverseloc}, the intersection of the images of $\phi_{s_{i,+},+}$ and $\phi_{s_{i,-},-}$ is not transverse at $\phi_{s_{i,+},+}(b_{i,+})=\phi_{s_{i,-},-}(b_{i,-})$. Thus, there are sequences of tangent vectors $\dot b_{i,+} \in T_{b_{i,+}}\mc{GC}_+(s_{i,+})$ and $\dot b_{i,-} \in T_{b_{i,-}}\mc{GC}_-(s_{i,-})$ such that 
\begin{equation}
\label{t2}
d_{b_{i,+}}\phi_{s_{i,+},+}(\dot b_{i,+})=d_{b_{i,-}}\phi_{s_{i,-},-}(\dot b_{i,-}).
\end{equation}
We can assume that the sequences $\dot b_{i,+}$, $\dot b_{i, -}$ are uniformly bounded. From the mentioned results of Tamburelli, $T_{b_{+}}\mc{GC}_+(s_{+})$, $T_{b_{-}}\mc{GC}_-(s_{-})$ are finite-dimensional.
Hence, by a combination of Lemma~\ref{tangspaces} and Lemma~\ref{subconverge}, these sequences subconverge to $\dot b_{+} \in T_{b_{+}}\mc{GC}_+(s_{+})$ and $\dot b_{-} \in T_{b_{-}}\mc{GC}_-(s_{-})$. Due to the definition of the maps $\phi_{s_\pm, \pm}$ via formula~(\ref{ks}), equation (\ref{t2}) implies that \[d_{b_{+}}\phi_{s_{+},+}(\dot b_{+})=d_{b_{-}}\phi_{s_{-},-}(\dot b_{-}).\] Hence, Lemma~\ref{transverseloc} implies that $d_e \mc I$ is not injective, which contradicts to Lemma~\ref{infrig}.
\end{proof}

\subsection{Local homeomorphism}
\label{ssc:homeo}

Now our goal is to prove the following statement:

\begin{prop}
\label{locrig}
Let $e \in \mc E_F$. There exists a neighborhood $U_e$ of $e$ in $\mc E$ such that the image of $U_e$ by $\mc I$ is open and $\mc I|_{U_e}$ is a homeomorphism onto its image.
\end{prop}

We need a preliminary work. Let $\Sigma$ be a spacelike surface in $\ads^3$, $n$ be the field of future-oriented unit normals, and $u$ be a vector field over $\Sigma$ (not necessarily tangent), which is decomposed as $u=v+an$ for a tangent vector field $v$ and a function $a$. Denote the first and the second fundamental forms of $\Sigma$ by $\fff$ and $\sff$ respectively, and denote by $\nabla$ the induced connection on $\Sigma$. It is a standard computation, to check that the derivative of $\fff$ along $u$ is given by (see, e.g.,~\cite{GHL})
\begin{equation}
\label{fffder}
\dot\fff(x, y)=2a\sff(x,y)+\fff(\nabla_x v, y)+\fff(x, \nabla_y v).
\end{equation}

The bundle $Z$ of symmetric $(0,2)$-tensors on $\Sigma$ has rank 3. Denote by $Z_1$ its subbundle of traceless tensors with respect to $\sff$ and denote $V_1:=T\Sigma$. Note that the elements of $Z_1$ are orthogonal to $\sff$ with respect to $\sff$. Furthermore, let $b$ be the shape operator, $j$ be the almost-complex structure of $\sff$ and $\nabla^\tff$ be the Levi--Civita connection of the third fundamental form. Define a differential operator $\mc P: C^\infty(V_1) \rar C^\infty(Z)$ by
\begin{equation}
\label{baroperator}
(\mc P v)(x,y):=\sff(\nabla^\tff_x (b^{-1}v)+j\nabla^\tff_{jx}(b^{-1}v), y).
\end{equation}
The significance of $\mc P$ is in that the kernel of $\mc P$ is exactly the space of the tangent components to the isometric vector fields on $\Sigma$.
A direct computation (see, e.g.,~\cite[Proposition 8.4]{Sch}) shows that its image is in $C^\infty(Z_1)$. Furthermore, it was proven in~\cite[Proposition 8.4]{Sch} that
\begin{lm}
\label{tantan}
For any vector field $u$ over $\Sigma$ with tangent component $v$, $\mc Pv$ is the $\sff$-orthogonal projection of $\dot \fff$ to $C^\infty(Z_1)$.
\end{lm}
Note that in~\cite{Sch} the author considers the case of surfaces in hyperbolic 3-space, however, the proof of this proposition is independent on the ambient geometry.

Denote by $Z_2$ the subbundle of $Z$ spanned by $\sff$. Let $V_2$ be the normal subbundle of $T\ads^3$ over $\Sigma$. Formula~(\ref{fffder}) establishes a natural isomorphism
\begin{equation}
\label{nornor}
Q: V_2 \rar Z_2.
\end{equation}

When $f: \tilde S \rar \ads^3$ is an equivariant spacelike embedding, we can pullback $\mc P$ to $S$. It has the same principal symbol as a similarly defined operator, obtained by replacing $\nabla^\tff$ with $\nabla^\sff$, the Levi--Civita connection of $\sff$. The latter operator has the same kernel and cokernel as the standard Cauchy--Riemann operator $\bar\pt$ (defined on de Rham complex tensored with $TS$). It follows from the Atiyah--Singer theorem (see, e.g.,~\cite[Section 2]{Gro}) that
\begin{lm}
\label{index}
For a closed surface $S$ of genus $k$, as an operator $C^\infty(TS) \rar C^\infty(Z_1)$, where the latter is now considered as a bundle over $S$, $\mc P$ is an elliptic operator of index $-(6k-6)$.
\end{lm}

Now we make another step aside. For the moment, let $e$ be an element of $\mc E_+$. Consider its lift $(\rho, f) \in \hat{\mc E}_+$. Pick a right inverse $\mu: H^1_\rho(  \mf g) \rar Z^1_\rho(  \mf g)$ to the projection $Z^1_\rho(  \mf g) \rar H^1_\rho(  \mf g)$. Recall that $\tilde V_f$ is the restriction of $T\ads^3$ to $f(\tilde S)$, which we now denote just by $\tilde V$, and that $C^\infty_\rho(\tilde V)$ is the space of automorphic vector fields with respect to $\dot\rho \in Z^1_\rho(  \mf g)$. Denote by $C^\infty_{\rho, n}(\tilde V)$ the subspace of automorphic vector fields with respect to $\dot\rho$ in the image of $\mu$. Recall that $T_{(\rho, f)}\hat{\mc E}_+\cong C^\infty_\rho(\tilde V)$. After considering the $G$-action, we can identify $T_{e}\mc E_+ \cong C^\infty_{\rho, n}(\tilde V)$. Denote by $V$ the projection of $\tilde V$ to $S$. There is a natural projection $\zeta:C^\infty_{\rho, n}(\tilde V) \rar H^1_\rho(  \mf g)$, whose kernel consists of equivariant vector fields, and hence can be identified with $C^\infty(V)$.  We pick a right inverse $\nu: H^1_\rho(\mf g) \rar C^\infty_{\rho, n}(\tilde V)$. The choices of $\mu$ and $\nu$ allow to identify
\begin{equation}
\label{equivar0}
T_e\mc E_+ \cong H^1_\rho(\mf g) \times C^\infty(V)
\end{equation}
via
\[T_e\mc E_+ \cong C^\infty_{\rho, n}(\tilde V)= {\rm Im}(\nu)\times\ker(\zeta)\cong H^1_\rho(\mf g) \times C^\infty(V).\]
Consider now a neighborhood $U$ of $e$ in $\mc E_+$ and a chart $\phi: U \rar T_e\mc E_+$ adapted for the projection $\pi_+:\mc E_+ \rar \mc R$. Here ``adapted'' means that for any $e' \in U$ and any $\tau \in H^1_\rho(\mf g)$ the composition 
\[d_{e'}\pi_+ \circ (d_{e'}\phi)^{-1}(\tau, 0)\]
is nonzero. It is possible to choose such a chart since $\pi_+$ is a submersion. We get an identification
\begin{equation}
\label{equivar}
TU \cong U \times H^1_\rho(\mf g) \times C^\infty(V).
\end{equation}

From now on we denote by $Z$ the bundle of symmetric $(0,2)$-tensor fields over $S$. For $e' \in U$ and $\tau \in H^1_\rho(\mf g)$ denote by $\xi(e', \tau) \in C^\infty(Z)$ the variation of the induced metric on $S$, where the metric is induced by $e'$ and the variation is induced by $(e', \tau, 0)$, considered as an element of $T_{e'} \mc E_+$ via identification~(\ref{equivar}). This can be seen as a smooth in $e'$ family of linear maps $\xi_{e'}: H^1_\rho(\mf g) \rar C^\infty(Z)$.

Every $e' \in U$ defines a decomposition of $V$ into two subbundles, the tangent one, and the normal one (with respect to the metric on $V$ induced by $e'$). We denote them by $V_{e', 1}$ and $V_{e',2}$ respectively. Define $V_1:=V_{e,1}$ and $V_2:=V_{e,2}$. Let $\alpha_{e',1}: V_{e',1} \rar V_1$ and $\alpha_{e', 2}: V_{e',2} \rar V_2$ be the orthogonal projections with respect to the metric on $V$ induced by $e$.

Also, every $e' \in U$ defines a decomposition of $Z$ into two subbundles, $Z_{e',1}$ and $Z_{e',2}$, where $Z_{e',2}$ is spanned by $\sff_{e'}$, the second fundamental form induced on $S$ by $e'$, and $Z_{e',1}$ is $\sff_{e'}$-orthogonal to it. Define $Z_1:=Z_{e,1}$ and $Z_2:=Z_{e,2}$. Let $\beta_{e',1}: Z_{e',1} \rar Z_1$ and $\beta_{e', 2}: Z_{e',2} \rar Z_2$ be the orthogonal projections with respect to the metric on $Z$ induced by $\sff_e$. For the map $\xi: U \times H^1_\rho(\mf g) \rar C^\infty(Z)$ defined above, denote its projections to $Z_1$ and to $Z_2$ by $\xi_1$ and $\xi_2$.

Now, however, we have to slightly change all this notation. Namely, we now pick $e \in \mc E$ rather than from $\mc E_+$. Then for some neighborhood $U$ of $e$ decomposition~(\ref{equivar}) will rather look like
\[TU \cong U \times H^1_\rho(\mf g) \times C^\infty(V_+) \times C^\infty(V_-),\]
where $V_+$ and $V_-$ are two pullbacks of $T\ads^3$ by the respective maps $f_+$ and $f_-$ for some lift $(\rho, f_+, f_-) \in \hat{\mc E}$ of $e$. To avoid redundant notation, we now define $V$ as $V_+ \times V_-$. Hence, now decomposition~(\ref{equivar}) holds for $e \in \mc E$, though in a slightly modified setting. 

The same relates to the bundle $Z$, which is now redefined as the product of two copies of the bundle of symmetric $(0,2)$-tensors over $S$. Similarly, we redefine in an obvious way all the decompositions and the projections introduced just above, as well as the function $\xi$. Now we can give a proof of Proposition~\ref{locrig}.

\begin{proof}
[Proof of Proposition~\ref{locrig}.]
Note that~(\ref{equivar0}) and~\cite[Theorem II.2.3.1]{Ham} imply that $\mc E_+$ and $\mc E$ are tame Fr\'echet manifolds. The same reference implies that so is $\mc S$.

By Proposition~\ref{inflocrig}, there exists a neighborhood $U_{e}$ of $e$ in $\mc E$ such that for each $e' \in U_e$, the differential $d_{e'}\mc I$ is injective. We assume further that $U_{e}$ is sufficiently small so that decomposition~(\ref{equivar}) holds and so that for every $e' \in U_e$ the bundle morphisms $\alpha_{e',1}$, $\alpha_{e'_2}$, $\beta_{e',1}$ and $\beta_{e', 2}$ are isomorphisms. Using~(\ref{equivar}) and decompositions $V=V_1 \times V_2$ and $Z=Z_1 \times Z_2$, we view the differential as a map
\[d\mc I: U_{e} \times H^1_\rho(\mf g) \times C^\infty(V_1)\times C^\infty(V_2) \rar C^\infty(Z_1) \times C^\infty(Z_2),\]
which we decompose as $d\mc I=\mc D_1 \times \mc D_2$ with
\[\mc D_1: U_{e} \times H^1_\rho(\mf g) \times C^\infty(V_1) \rar C^\infty(Z_1),\]
\[\mc D_2: U_{e} \times H^1_\rho(\mf g) \times C^\infty(V_2) \rar C^\infty(Z_2)\]
as follows.

Every $e' \in U_{e}$ determines a differential operator $\mc P_{e'}': C^\infty(V_{e',1}) \rar C^\infty(Z_{e',1})$ given by~(\ref{baroperator}). Define the differential operator 
\[\mc P_{e'}:=\beta_{e',1}\circ \mc P_{e'}'\circ\alpha_{e',1}^{-1}: C^\infty(V_{1}) \rar C^\infty(Z_{1}).\]
Then we define 
\[\mc D_1(e', \tau, v):=\mc P_{e'} v + \xi_1(e', \tau).\]

Also every $e' \in U_{e}$ determines an isomorphism of bundles $Q'_{e'}: V_{e', 2} \rar Z_{e',2}$ by~(\ref{nornor}). Define the isomorphism of bundles
\[Q_{e'}:=\beta_{e',2}\circ Q_{e'}'\circ\alpha_{e',2}^{-1}: V_{2} \rar Z_{2}.\]
Then define
\[\mc D_2(e', \tau, v):=Q_{e'} v + \xi_2(e', \tau).\]
From Lemma~\ref{tantan} and the definition of $\xi$, it follows that indeed $\mc D_1 \times \mc D_2 = d\mc I$.

By Lemma~\ref{index}, each operator $P_{e'}$ is elliptic of index $-(12k-12)$. By assumption, for each $e' \in U_{e}$, the map $d_{e'} \mc I$ is injective, hence, so is $\mc D_{1, e'}:=\mc D_1(e', \cdot, \cdot)$. It is well-known that ${\rm dim}\mc R_F=6k-6$, see, e.g.~\cite{FM}. Because $H^1_\rho(\mf g)$ is the tangent space to $\mc R \cong \mc R_F \times \mc R_F$, we have ${\rm dim}H^1_\rho(\mf g)=12k-12$. Hence, the index of $\mc D_{1, e'}$ is zero. It follows that for each $e' \in U_{e}$, the map $\mc D_{1, e'}$ is an isomorphism. Now by~\cite[Theorem II.3.3.3]{Ham}, the inverses to $\mc D_{1, e'}$ constitute a smooth tame family of linear maps.

For each $e' \in U_{e}$ and $\tau \in H^1_\rho(\mf g)$, $\mc D_2(e', \tau, \cdot)$ is induced by a bundle diffeomorphism $V_2 \rar Z_2$, which depends smoothly on $e'$ and $\tau$. Altogether this means that the inverses to $d\mc I$ constitute a smooth tame family of linear maps.

It follows that we can apply the Nash--Moser inverse function theorem~\cite[Theorem III.1.1.1]{Ham}, which implies that $\mc I$ is a local homeomorphism around $e$ onto an open set.
\end{proof}

\subsection{Proof of the main result}
\label{ssc:main**}

Now we can prove Theorem~\ref{main**}.

\begin{proof}[Proof of Theorem~\ref{main**}.]
Both spaces $\mc E$ and $\mc S \times \mc S$ are nuclear, hence paracompact~\cite[Theorem 16.10]{KM}. Since they also are Hausdorff and locally metrizable, they are metrizable by the Smirnov metrization theorem~\cite[Theorem 42.1]{Mun}. Choose a metric $d_S$ on $\mc S\times \mc S$ and consider its restriction to $\mc S_F:=\mc I(\mc E_F)$. By Labourie--Schlenker~\cite{LS}, $\mc I|_{\mc E_F}$ is a homeomorphism onto the image, we can pullback $d_S$ from $\mc S_F$ to $\mc E_F$. The subset $\mc E_F$ is closed in $\mc E$ and it is a result of Hausdorff~\cite{Hau} that a metric on $\mc E_F$ extends to a metric $d_E$ on $\mc E$ compatible with the topology of~$\mc E$. 

Now from Proposition~\ref{locrig}, for every $e \in \mc E_F$ there is $r(e)>0$ such that $\mc I$ is a homeomorphism onto the image when restricted to the closed ball $B(e, r(e))$. We can pick a neighborhood $U_e$ of $e$ such that $U_e \subset B(e, r(e)/3)$ and $\mc I(U_e) \subset B(\mc I(e), r(e)/3)$. 
%Since its image is open in $\mc S$, it contains an open ball $B(\mc I(e), r_S(e))$ such that $r_S(e)<r_E(e)/3$ and $B(\mc I(e), r_S(e) \subseteq \mc I(e, r_E(e)/3)$. Consider $U_e:=\mc I^{-1}(B(\mc I(e), r'(e))$, which is open in $\mc E$. 
We claim that $\mc I$ is injective on $U=\bigcup_{e \in \mc E_F} U_e$, which is then a desired neighborhood of $\mc E_F$. Indeed, suppose that $h=\mc I(e_1')=\mc I(e_2')$ for $e_1', e_2' \in U$, $e_1' \in U_{e_1}$, $e_2' \in U_{e_2}$ with $e_1, e_2 \in \mc E_F$, and that $r_1:=r(e_1) \leq r_2:=r(e_2)$. We claim that then $e_1' \in B(e_2, r(e_2))$. Indeed,
\[d_E(e_1', e_2) \leq d_E(e_1', e_1)+d_E(e_1,e_2)\leq r_1/3+d_S(\mc I(e_1),\mc I(e_2))\leq \]
\[\leq r_1/3+d_S(\mc I(e_1), h)+d_S(h,\mc I(e_2))\leq r_1/3+r_1/3+r_2/3\leq r_2.\]
This means that both $e_1', e_2' \in B(e_2, r(e_2))$, which is a contradiction, as $\mc I$ is injective when restricted to $B(e_2, r(e_2))$.
\end{proof}

 %% rigidity near fuchsian

\section{Spacetimes with polyhedral boundary}
\label{sc:polyhedral}

We now prove the polyhedral version of our result, Theorem~\ref{mainp*}. In this case the proof is quite parallel to the smooth one, thus we will give a bit less details in this setting.

%Let $V \subset \pt M$ be a finite set with at least one point in every component of $\pt M$, $\ms M(V)$ be the space of globally hyperbolic anti-de Sitter metrics on $M$ with (convex) strictly polyhedral spacelike boundary and with vertices at $V$, considered up to isotopy. (The involved notions and the topology on $\ms M(V)$ are discussed below.) Denote by $\ms M_F(V) \subset \ms M(V)$ the subspace of Fuchsian metrics, which are defined similarly to the smooth setting. 
%We establish a proof of the following result.
%
%\begin{thm}
%\label{mainp}
%There exists a neighborhood $U$ of $\ms M_F(V)$ in $\ms M(V)$ such that for any $m_1, m_2 \in U$, if there exists $f: (M, m_1) \rar (M, m_2)$ isotopic to the identity such that $f|_{\pt M}$ is an isometry, then $f$ is isotopic to an isometry.
%\end{thm}
%
%Let $\ms S(V)$ be the space of hyperbolic cone-metrics on $\pt M$ with cone-angles $>2\pi$ and vertices at $V$, considered up to isotopies on $M$. We have the induced metric map $\ms I_V:\ms M(V) \rar \ms S(V)$. Theorem~\ref{mainp} is the rigidity part of the following statement. 
%
%\begin{thm}
%\label{mainp*}
%There exist a neighborhood $U$ of $\ms M_F(V)$ in $\ms M(V)$ and an open subset $U'\subset \ms S(V)$ such that $\ms I_V|_U$ is a homeomorphism from $U$ to $U'$.
%\end{thm}

\subsection{Anti-de Sitter spacetimes with marked points}

Every GHC AdS (2+1)-spacetime $M$ admits a canonical extension to a GHMC AdS (2+1)-spacetime, which we denote for now by $N$. The spacetime $N$ has a \emph{convex core}, which can be defined, e.g., as the inclusion-minimal closed totally convex subset of $N$. See details, e.g., in~\cite{BS2}. In particular, the convex core is contained in $M$ (considered here as a subset of $N$). We say that $M$ has \emph{strictly polyhedral} boundary if it is locally modeled on convex polyhedral subsets of $\ads^3$ and the boundary does not intersect the convex core. The last condition ensures a reasonable behavior of the combinatorics of the boundary as we vary the metric on $M$. See~\cite[Section 3]{Pro2} on the description of this in the hyperbolic case, as the anti-de Sitter case is the same in this regard. By a \emph{vertex} of $M$ we mean a point on $\pt M$ at which $M$ is locally modeled on a convex cone different from a half-space or from the wedge of two half-spaces.

The space $\ms M(V)$, defined in the introduction, is endowed with the topology of uniform convergence of developing maps on compact subsets of $\tilde M$. As well as in the smooth case, we actually care only about the structures on the boundary. However, in this case the objects should be thought in a slightly different fashion.

Pick $\rho \in \hat{\mc R}$, recall that $D(\rho) \subset \ads^3$ is the maximal open $\rho$-invariant convex set, on which the action of $\rho$ is free and properly discontinuous. It also contains the \emph{convex core} $C(\rho)$, which is the inclusion-minimal closed $\rho$-invariant convex set. Alternatively, it can be defined as the convex hull of the limit set $\Lambda(\rho)$ at the boundary at infinity of $\ads^3$. (Recall that the $\rho$-quotient of $D(\rho)$ is a GHMC AdS spacetime. Then the convex core of this spacetime, mentioned above, is exactly the quotient of $C(\rho)$.) There are two connected components of $D(\rho)\backslash C(\rho)$, corresponding to the future-convex and past-convex components of $\pt D(\rho)$. We denote these components of $D(\rho)\backslash C(\rho)$ by $D_+(\rho)$ and $D_-(\rho)$.

For now, let $V \subset S$ be a finite subset and $\tilde V \subset \tilde S$ be its full preimage, equipped with the natural $\pi_1S$ action. Let $\rho \in \hat{\mc R}$ and let $f: \tilde V \rar D_+(\rho)$ be a $\rho$-equivariant map. Denote by $\clconv(f)$ the closed convex hull of $f(\tilde V)$ in $\ads^3$. Only one of its boundary components has nonempty intersection with $D_+(\rho)$. This component is spacelike and future-convex. We denote it by $\Sigma(f)$. We say that $f$ is \emph{future-convex} if no point of $f(\tilde V)$ is a convex combination of others and \emph{strictly future-convex} if, in addition, $\Sigma(f)$ is disjoint from $C(\rho)$. Note that then $\Sigma(f) \subset D_+(\rho)$, $f(\tilde V) \subset \Sigma(f)$ and $\Sigma(f)$ has a a decomposition into vertices, edges and faces, where the set of vertices is exactly $f(\tilde V)$, every edge is a geodesic segment between two vertices, and every face is isometric to a convex hyperbolic polygon. For more details on this decomposition, we refer to~\cite[Section 3.4.5]{Pro4}.

Denote the space of quasi-Fuchsian equivariant strictly future-convex maps by $\hat{\ms E}_+(V)$. By picking a lift in $\tilde V$ for each $v \in V$, the space $\hat{\ms E}_+(V)$ can be considered as a subset of $\hat{\mc R} \times (\ads^3)^V$, from which it inherits a topology. This is an open subset, hence $\hat{\ms E}_+$ becomes a smooth manifold. Denote by $\ms E_+(V)$ its quotient by $G$. It is a smooth manifold of dimension $6k-6+3n$, where $k$ is the genus of $S$ and $n:=|V|$. Similarly, we define quasi-Fuchsian equivariant strictly past-convex maps and denote the respective spaces by $\hat{\ms E}_-(V)$ and $\ms E_-(V)$.

\begin{rmk}
We use here the notation $\ms E$ rather than $\mc E$, following our convention to use the first font for spaces up to isotopy (cf. the smooth case). An element $f \in \hat{\ms E}_+(V)$ extends to an equivariant homeomorphism from $\tilde S$ onto $\Sigma(f)$, but it is determined only up to isotopy. There is a way to consider polyhedral equivariant embeddings of $S$ as the main objects, but some things are more peculiar in such setting.
\end{rmk}

A tangent vector to $\hat{\ms E}_+(V)$ at $(\rho, f)$ is a vector field $v$ on $f(\tilde V)$ satisfying the automorphicity condition~(\ref{automorph}). Again, we say that $v$ is \emph{trivial} if it is the restriction of a global Killing field. Pick an arbitrary triangulation of $\Sigma(f)$ subdividing its face decomposition. We say that $v$ is \emph{isometric} if it induces zero variation on each edge-length of the triangulation. Note that this is independent on the chosen triangulation.

Let $\ms S(V)$ be the space of isotopy classes of hyperbolic cone-metrics on $S$ with cone-angles $>2\pi$ with vertices at $V$. Here we mean that every point of $V$ has cone-angle $>2\pi$ in such metrics. Every such metric admits a geodesic triangulation and $\ms S(V)$ is endowed with a topology of a smooth manifold of dimension $6k-6+3n$ by considering as charts the edge-length maps for varying triangulations. See, e.g.,~\cite[Section 2.2]{Pro3} for details. 

%Now we consider two finite subsets $V_+, V_- \subset S$ and redefine $V:=\{V_+, V_-\}$. Denote by $\hat{\ms E}(V)$ the subset of $\hat{\ms E}_+(V_+)\times\hat{\ms E}_-(V_-)$ consisting of the pairs of maps that are equivariant with respect to the same representation, and by $\ms E(V)$ denote its quotient by $G$. 

Return to the space $\ms M(V)$ defined in the introduction. Denote by $V_+$ and $V_-$ the intersections of $V$ with both boundary components of $M$, considered as subsets of $S$ via an identification $M \cong S \times [-1, 1]$, and redefine $V:=\{V_+, V_-\}$. 
Denote by $\hat{\ms E}(V)$ the subset of $\hat{\ms E}_+(V_+)\times\hat{\ms E}_-(V_-)$ consisting of the pairs of maps that are equivariant with respect to the same representation, and by $\ms E(V)$ denote its quotient by $G$.
We have a natural continuous injective map $\ms M(V) \hookrightarrow \ms E(V)$. Imitating the proof from the smooth case, we obtain that this map is a homeomorphism. We denote by $\ms E_F(V)$ the subspace of Fuchsian pairs of maps, i.e., the ones that have a lift in $\hat{\ms E}_F$ with representation into $G_F$.

Pick $e \in \ms E(V)$ represented by $(\rho, f_+, f_-) \in \hat{\ms E}(V)$. We have two associated $\rho$-invariant spacelike surfaces $\Sigma(f_+)$ and $\Sigma(f_-)$. They are endowed with hyperbolic cone-metrics with cone-angles $>2\pi$. Consider their $\rho$-quotients. We have a canonical isomorphism of the fundamental group of each of these quotients with $\pi_1 S$. This allows to define  homeomorphisms from them to $S$, determined up to isotopy. Furthermore, the markings of vertices given by $f_+$ and $f_-$ allow to associate the induced metrics with elements of $\ms S(V_+)$ and $\ms S(V_-)$ respectively. This defines the induced metric map
\[\ms I_{V}: \ms E(V) \rar \ms S(V_+)\times \ms S(V_-).\]
The respective reformulation of our polyhedral result is

\begin{thm}
\label{main**p}
There exists a neighborhood $U$ of $\ms E_F(V)$ in $\ms E(V)$ and an open subset $U'\subset \ms S(V_+)\times \ms S(V_-)$, which contains the diagonal, such that $\ms I_{V}|_U$ is a homeomorphism onto $U'$
\end{thm}

The proof is based on the following facts. 

\begin{lm}
\label{c1}
The map $\ms I_{V}$ is $C^1$.
\end{lm}

The proof is identical to the proof of respective results in other settings, see, e.g.,~\cite[Lemma 2.13]{FP} for a Minkowski version or~\cite[Lemma 2.17]{Pro2} for a de Sitter version.

\begin{prop}
\label{infrigp}
Let $e \in \ms E_F(V)$. Then $d_e\ms I_V$ is injective.
\end{prop}

The proof is given in the next subsection.

\begin{thm}
\label{Fil}
The restriction $\ms I_V|_{\ms E_F(V)}$ is a homeomorphism onto the image.
\end{thm}

This is the content of article~\cite{Fil2} of Fillastre.

\begin{proof}[Proof of Theorem~\ref{main**p}.]
Because $\ms E(V)$ and $\ms S(V_+)\times \ms S(V_-)$ are finite-dimensional manifolds of the same dimension, Proposition~\ref{infrigp} actually means that for each $e \in \ms E_F(V)$, $d_e\ms I_V$ is an isomorphism. Together with Lemma~\ref{c1}, the inverse function theorem implies that $\ms I_V$ is a local homeomorphism around $\ms E_F(V)$. Using Theorem~\ref{Fil} instead of Theorem~\ref{LS}, the proof is finished in the same way as the proof of Theorem~\ref{main**}.
\end{proof}

\subsection{Minkowski spacetimes with marked points and infinitesimal rigidity}

In this section we prove Proposition~\ref{infrigp}. As in the smooth case, the proof is done by using the Pogorelov map. We will rely on~\cite{FP} by Fillastre--Prosanov instead of~\cite{Smi}. We now recall the basics of polyhedral surfaces in Minkowski 3-space $\R^{2,1}$. We adapt some notions from the previous subsection to this setting.

Pick $\rho \in \hat{\mc R}^0$. Associated with it are two maximal open $\rho$-invariant convex sets, on which the action of $\rho$ is free and properly discontinuous, one is future-complete and one is past-complete. We denote them by $D_+^0(\rho)$ and $D_-^0(\rho)$. See~\cite{Mes} of Mess for details.

Again, let $V \subset S$ be a finite subset and $\tilde V \subset \tilde S$ be its full preimage, equipped with the natural $\pi_1S$ action. Let $\rho \in \hat{\mc R}^0$ and let $f: \tilde V \rar D_+^0(\rho)$ be a $\rho$-equivariant map. Denote by $\clconv(f)$ the closed convex hull of $f(\tilde V)$. In this setting, it is a future-complete set with one boundary component. Furthermore, this component is future-convex and spacelike. We denote it by $\Sigma(f)$. We say that $f$ is \emph{future-convex} if no point of $f(\tilde V)$ is a convex combination of others. (An interesting difference with the anti-de Sitter situation is that here is no notion of strict convexity for our maps due to the lack of convex cores in this setting.) Note that then $f(\tilde V) \subset \Sigma(f)$. The surface $\Sigma(f)$ also has a a decomposition into vertices, edges and faces, where the set of vertices is exactly $f(\tilde V)$, every edge is a geodesic segment between two vertices, and every face is isometric to a convex Euclidean polygon. 

Denote the space of quasi-Fuchsian equivariant future-convex maps by $\hat{\ms E}_+^0(V)$. By picking a lift in $\tilde V$ for each $v \in V$, the space $\hat{\ms E}_+(V)$ can be considered as a subset of $\hat{\mc R}^0 \times (\R^{2,1})^V$, from which it inherits a topology. This is an open subset, hence $\hat{\ms E}_+^0(V)$ becomes a smooth manifold. Denote by $\ms E_+^0(V)$ its quotient by $G$. It is a smooth manifold of dimension $6k-6+3n$, where $k$ is the genus of $S$ and $n:=|V|$. Similarly, we define quasi-Fuchsian equivariant past-convex maps and denote the respective spaces by $\hat{\ms E}_-^0(V)$ and $\ms E_-^0(V)$. The notions of automorphic, trivial and isometric vector fields extend to this setting accordingly.

Let $\ms S^0(V)$ be the space of Euclidean cone-metrics with cone-angles $>2\pi$ and vertices at $V$, considered up to isotopy. For $e \in \ms E_+^0(V)$ lifting to $(\rho, f)$, the induced metric on $\Sigma(f)$ can be associated with an element in $\ms S^0(V)$. This defines the induced metric map
\[\ms I_{V, +}^0: \ms E_+^0(V) \rar \ms S^0(V).\]
It was shown in~\cite[Lemma 2.13]{FP} that this map is $C^1$. Moreover, the authors proved

\begin{thm}
\label{FP1}
The map $\ms I_{V, +}^0$ is a surjective submersion.
\end{thm}

This is a direct corollary of~\cite[Theorem I]{FP}. See also the discussion in~\cite[Section 4.1]{FP}. For $s \in \ms S^0(V)$ denote $(\ms I_{V, +}^0)^{-1}(s)$ by $\ms E_+^0(s)$. Due to Theorem~\ref{FP1}, this is a $C^1$-submanifold of $\ms E_+^0(V)$. We now define a map
\[\phi_{s,+}^0: \ms E^0_+(s) \rar T\ms T,\]
which sends a $\rho$-equivariant map to $\rho$ (recall that $\mc R^0 \cong T\ms T$). It is also $C^1$, see~\cite[Section 4.1]{FP}. We define similarly the space $\ms E^0_-(s)$ and the map $\phi^0_{s,-}$. It is another direct corollary of~\cite[Theorem I]{FP} that

\begin{thm}
\label{FP2}
For any Euclidean cone-metric $s$ with cone-angles $>2\pi$, the compositions of $\phi_{s, \pm}^0$ with the projection to $\ms T$ are $C^1$-diffeomorphisms.
\end{thm}

In particular, the maps $\phi^0_{s, \pm}$ are $C^1$-immersions. 
%Furthermore, it follows from Theorem~\ref{FP2} that $h$ determines via $\phi_{h,\pm}^0$ two vector fields on $\ms T$. We denote the one defined via $\phi_{h,+}^0$ by $x_h$. Then the second one is just $-x_h$. 
From~\cite[Theorem II]{FP}, we have

\begin{thm}
\label{FP3}
Pick two finite subsets $V_+, V_- \subset S$ and metrics $s_+ \in \ms S^0(V_+)$, $s_- \in \ms S^0(V_-)$. Then the intersection of $\phi^0_{s_+,+}$ and $\phi^0_{s_-,-}$ is unique and transverse.
\end{thm}

Note that in~\cite{FP} Theorem~\ref{FP3} is formulated in the language of vector fields given by Theorem~\ref{FP2}.
Now we can show

\begin{lm}
\label{infrigminkp}
Pick two finite subsets $V_+, V_- \subset S$.
Consider $e_+=(\rho, f_+) \in \ms E_+^0(V_+)$, $e_-=(\rho, f_-) \in \ms E_-^0(V_-)$ and $\dot \rho \in Z^1_\rho(\mf g^0)$. Let $v$ be a $\dot\rho$-automorphic isometric vector field on $f_+(\tilde V_+)$ and $f_-(\tilde V_-)$. Then $v$ is trivial.
\end{lm} 

\begin{proof}
Let $s_+$ and $s_-$ be the induced metrics.
If $v$ is nontrivial, then it defines a nonzero tangent vector $\dot e_+$ to $\ms E_+(s_+)$ at $s_+$ and $\dot e_-$ to $\ms E_-(s_-)$ at $e_-$. We then have
\begin{equation}
\label{tp}
d_{e_+}\phi^0_{s_+,+}(\dot e_+)=d_{e_-}\phi^0_{s_-,-}(\dot e_-).
\end{equation}
Due to Theorem~\ref{FP2}, expression~(\ref{tp}) is nonzero. But this contradicts Theorem~\ref{FP3}.
\end{proof}

We can now prove Proposition~\ref{infrigp}.

\begin{proof}[Proof of Proposition~\ref{infrigp}.]
As in the proof of Lemma~\ref{infrig}, we exploit the simultaneous model for anti-de Sitter and Minkowski 3-spaces as well as the Pogorelov map. Recall the model from Section~\ref{secmink}, particularly the point $o$ and the set $C$. Pick $e \in \ms E_F(V)$ and lift it to a triple $(\rho, f_+, f_-) \in \hat{\ms E}_F(V)$, where $\rho: \pi_1 S \rar G_F$ and $G_F$ fixes $o$. Then the images of $f_+$ and $f_-$ are contained in $C$. Pick $\dot e \in T_e\ms E(V)$ such that $d_e\ms I_V(\dot e)=0$. Then it is represented by a $\dot\rho$-automorphic isometric vector field on $f_+(\tilde V_+)$ and $f_-(\tilde V_-)$ for $\dot\rho\in Z^1_\rho(\mf g)$.

Recall the Pogorelov map $\Phi$ and the associated map $\Psi_\rho$. From Lemma~\ref{pogaut}, $\Phi(v)$ is a $\Psi_\rho(\dot\rho)$-automorphic vector field. Also, if $v$ induces a zero variation on the distance between two points, then this means that its restriction to these two points coincides with the restriction of a Killing field. (Indeed, by subtracting a Killing field, one can assume that at one point $v$ is zero. Then at the second point $v$ is orthogonal to the segment between the points and any such $v$ is induced by an infinitesimal rotation around the first point.) Hence, Lemma~\ref{ipm} implies that $\Phi(v)$ is isometric. By Lemma~\ref{infrigminkp}, it is trivial, thereby so is $v$. It follows that $\dot e=0$.
\end{proof}
 %% polyhedral

\bibliographystyle{abbrv}
\bibliography{AdsRigidity}

\end{document}